\documentclass{amsart}

\usepackage{mathrsfs} 
\usepackage{amsfonts}
\usepackage{amscd}
\usepackage{amssymb,latexsym,amsmath,amscd} 
\usepackage{enumerate}
\usepackage[english]{babel}
\usepackage[utf8]{inputenc}
\usepackage{mathrsfs}
\usepackage{amsmath}
\usepackage{amsthm}
\usepackage{amssymb}
\usepackage{indentfirst}
\usepackage{color}
\usepackage[dvips]{graphicx}
\usepackage[matrix,arrow,tips,curve]{xy}
\usepackage{ifthen}

\theoremstyle{plain} 
\newtheorem{prop}{Proposition}[section]
\newtheorem*{prop*}{Proposition}

\newtheorem{thm}[prop]{Theorem}
\newtheorem*{thm*}{Theorem}
\newtheorem{lem}[prop]{Lemma} 
\newtheorem{cor}[prop]{Corollary}

\theoremstyle{remark}

\newtheorem{oss}[prop]{Remark}
\newtheorem{ex}[prop]{Example}

\newtheorem{ass}[prop]{Setting}
\theoremstyle{definition}
\newtheorem{defn}[prop]{Definition}
\newcommand{\hk}{hyperk\"{a}hler }

\newcommand{\kahl}{K\"{a}hler }

\newcommand{\kntiposp}{$K3^{[n]}$ type }

\newcommand{\be}{\begin{equation}}
\newcommand{\ee}{\end{equation}}
\renewcommand{\phi}{\varphi}

\DeclareMathOperator{\Aut}{Aut}

\newcommand{\ZZ}{\mathbb{Z}}
\begin{document}
\title{Monodromy and birational geometry of O'Grady's sixfolds}
\author{Giovanni Mongardi}
\address{Alma Mater Studiorum, Universit\`{a} di Bologna,  P.zza di porta san Donato, 5, 40126 Bologna, Italia}
\author{Antonio Rapagnetta}
\address{Università degli studi di Roma ``Tor Vergata'', Via della Ricerca Scientifica 1, 00133 Roma, Italia}

\begin{abstract}
We prove that the bimeromorphic class of a \hk manifold deformation equivalent to O'Grady's six dimensional one is determined by the Hodge structure of its Beauville-Bogomolov lattice by showing that the monodromy group is maximal. As applications, we give the  structure for the K\"ahler and the birational K\"ahler  cones in this deformation class and we prove that the existence of a square zero divisor implies the existence a rational lagrangian fibration with fixed fibre types. 

\end{abstract}

\maketitle

\section{Introduction}

This paper deals with a deformation class of \hk manifolds, which was first discovered by O'Grady \cite{OG6}. 
These manifolds are sixfolds whose second Betti number is $8$ and are usually called manifolds of $OG6$ type. 
Manifolds in this family are obtained in two ways. The first construction, is obtained by taking a generic  abelian surface and a Mukai vector $w$ of square $2$. The moduli space of Gieseker semistable sheaves with Mukai vector $2w$ is a singular tenfold with rational singularities, whose Albanese fibre admits a crepant resolution that is
a \hk manifold in the family  we are dealing with. This was proven by\textsl{}
O'Grady \cite{OG6} for a special Mukai vector. Later M. Lehn and Sorger \cite{ls} showed that, under our assumption on $w$, the blow up of the Albanese fibre of the moduli space along its singular locus always gives a crepant resolution and Perego and the second named author proved in \cite{pr_crelle} that these crepant resolutions are deformation equivalent, along smooth projective deformations, to the original O'Grady's example.\\
A second construction was obtained in \cite{MRS}, by considering a principally polarized abelian surface $A$ and its Kummer K3 surface $S$. On a moduli space of sheaves on $S$, the authors construct a non regular involution, whose quotient is birational to a manifold of $OG6$ type. This last construction was used to compute the Hodge numbers of manifolds of $OG6$ type, and the present paper is a continuation of \cite{MRS} aiming at a greater understanding of their geometry.\\

For every \hk manifold $X$, the second cohomology group $H^{2}(X,\ZZ)$ has a natural lattice structure induced by the Beauville-Bogomolov form $B_{X}$, compatible with the weight two Hodge structure and this datum is a fundamental invariant for $X$: by Verbitsky's Global Torelli Theorem
it determines, up to a finite indeterminacy, the bimeromorphic class of $X$
among  \hk manifolds in the same \hk deformation equivalence class of $X$.

In this paper we show that in the case of $OG6$ some basic geometric invariants are completely determined by the weight two Hodge structure of the Beauville-Bogomolov lattice.   
The main result is Theorem \ref{thm:monodromy}(2) stating that 
the Classical Bimeromorphic  Global Torelli Theorem, as conjectured in  Speculation 10.1 of \cite{huy_basic}, holds  for manifolds of $OG6$ type:
\begin{thm}
Let $X,Y$ be two \hk  manifolds of $OG6$ type and equip the cohomology groups $H^2(X,\ZZ)$ and $H^2(Y,\ZZ)$ with the lattice structures induced by the Beauville-Bogomolov forms of 
$X$ and $Y$. Then $X$ and $Y$ are bimeromorhic if and only if there exists an isometric isomorphism of integral Hodge structures $H^2(X,\ZZ)\cong H^2(Y,\ZZ)$.
\end{thm}

We remark that the Classical Bimeromorphic Global Torelli Theorem  rather rarely happens to hold  for deformation equivalence classes of known \hk manifolds: it only holds (among known \hk manifolds)  for $K3$ surfaces and their Hilbert schemes of $n$ points if $n-1$ is a prime power and for O'Grady's ten dimensional manifolds thanks to recent work of Onorati \cite{ono}. The Theorem fails for Hilbert schemes on K3 surfaces if $n$ is not a prime power.
Finally, the Classical Bimeromorphic Global Torelli Theorem always fails for generalized Kummer manifolds, due to a classical counterexample by Namikawa \cite{nami_killtor}, as replacing the abelian surface used to construct the generalized Kummer manifold with its dual does not change the second Hodge structure but the two Kummer manifolds are not  birational in general. Due to this counterexample and due to the role of an abelian surface in the construction of O'Grady's six dimensional manifolds, one could expect a similar failure of the Global Torelli Theorem for O'Grady's sixfolds. However, this is not the case and an intuitive explanation of this fact could be that the relevant manifolds for O'Grady's construction are not only an abelian surface $A$, but rather $A\times A^\vee$ and the Kummer K3 surface $\widetilde{A_{/\pm 1}}$ (cf. \cite{MRS}).

We present two main applications. The first one is Theorem \ref{thm:amp} and gives the description of 
the K\"ahler cone $K(X)\subset H^{1,1}(X,\mathbb{R})$ and of  the closure $\overline{BK}(X)\subset H^{1,1}(X,\mathbb{R})$ of the  birational K\"ahler cone of a \hk manifold of $OG6$ type $X$, in terms of its weight two Hodge structure in a purely lattice theoretic way.
\begin{thm}
Let $X$ be a \hk manifold of $OG6$ type and 
let the positive cone $C(X)$ of $X$ be the connected component of the cone $$\left\{\alpha\in H^{1,1}(X,\mathbb{R}): B_{X}(\alpha,\alpha)>0 \right\}$$
containing a K\"ahler class. Then
\begin{enumerate}
\item{The closure in $H^{1,1}(X,\mathbb{R})$ of the birational K\"ahler cone $\overline{BK}(X)$ of $X$ is  the closure of the  connected component of $$C(X)\setminus \bigcup_{\substack{\alpha\in H^{1,1}(X,\ZZ),\\
B_{X}(\alpha,\alpha)=-2\; or \; -4, \\ div(\alpha)=2.}} \alpha^{\perp_{B_{X}}}$$ containing a K\"ahler class.}
\item{The K\"ahler cone $K(X)$ is the connected component of 
$$C(X)\setminus \bigcup_{\substack{\alpha\in H^{1,1}(X,\ZZ),\\
B_{X}(\alpha,\alpha)=-2\; or  \\ 
B_{X}(\alpha,\alpha)=-4\; and\;  div(\alpha)=2.}} \alpha^{\perp_{B_{X}}}$$
containing a K\"ahler class.}
\end{enumerate}
\end{thm}
In this statement, for $\alpha\in H^{2}(X,\ZZ)$ the subspace $\alpha^{\perp_{B_{X}}}\subseteq H^{2}(X,\mathbb{R})$ is the perpendicular to $\alpha$ with respect to the real extension of $B_{X}$ and $div(\alpha)$ is the divisibility of $\alpha$ in $ H^{2}(X,\ZZ)$, i.e. the minimum strictly positive integer that can be obtained as $B_{X}(\alpha,\beta)$ for  $\beta\in H^{2}(X,\ZZ)$ (see Definition \ref{def1}). We remark that, when $X$ is projective, the ample cone is the intersection of the K\"ahler cone with $H^{1,1}(X,\ZZ)\otimes\mathbb{R}$ and the movable cone is $\overline{BK(X)} \cap (H^{1,1}(X,\ZZ)\otimes\mathbb{R})$.

The second main application concerns the existence of lagrangian fibrations on \hk manifolds of $OG6$ type. A conjecture due to Hassett-Tschinkel, Huybrechts and Sawon predicts the existence of a lagrangian fibration on every \hk manifold admitting a non zero divisor  class  isotropic with respect to the Beauville-Bogomolov form.
Corollary \ref{cor:lagr} settles this conjecture for \hk manifolds of $OG6$ type.
\begin{cor}\label{cor:intro_lagr}
Let $X$ be a manifold of $OG6$ type with a divisor whose class is non  zero and isotropic. Then $X$ has a bimeromorphic  model which has a dominant map to $\mathbb{P}^3$ whose general fiber is a $(1,2,2)$-polarized abelian threefold.
\end{cor}
As a consequence of this Corollary, we also prove Beauville's weak splitting property for \hk manifolds of $OG6$ type with a square zero  non trivial line bundle (see Corollary \ref{ultimo}). 

Our proof of the Classical Bimeromorphic Global Torelli for $OG6$ 
rests on Markman's Hodge theoretic version \cite[Thm 1.3]{mark_tor} of Verbitsky's Global Torelli theorem \cite{ver_tor} stating that  
if  $X$ and $Y$ are two \hk manifolds in the same deformation equivalence class, they are bimeromorphic if and only if there exists an isometric, with respect to the Beauville-Bogomolov forms,  isomorphism  $H^2(X,\ZZ)\cong H^2(Y,\ZZ)$ coming from parallel transport. 
This Theorem reduces the problem of the validity of 
the Classical Bimeromorphic Global Torelli Theorem to the computation of 
the monodromy group, which is the group of all transformations of the second cohomology which can be obtained by taking parallel transport along loops in families of smooth K\"ahler deformations of our \hk manifold in the chosen deformation equivalence class.
The result of this calculation  (see Theorem \ref{thm:monodromy}(1)) can be considered the core of this paper. 
\begin{thm}
Let $X$ be a \hk manifold of $OG6$ type, then the Monodromy group $Mon^2(X)$
is the subgroup $O^{+}(H^{2}(X,\ZZ))$  of $\ZZ$ linear automorphisms
preserving the Beauville-Bogomolov form and the orientation of the positive cone $C(H^2(X,\ZZ))$  of $H^2(X,\ZZ)$.
\end{thm}
The  explanation of the definitions of $O^{+}$ and $C(H^2(Y,\ZZ))$ is given in Subsection \ref{ssec:lattice} and Remark \ref{positive cone}. Since every isometry of  
$H^{2}(X,\ZZ)$ is contained, up to sign, in $O^{+}(H^{2}(Y,\ZZ))$, every Hodge isometry $H^2(X,\ZZ)\cong H^2(Y,\ZZ)$ comes, up to sign, from parallel transport 
and the Classical Bimeromorphic Global Torelli Theorem holds for $OG6$ type manifolds.\\  
To contextualize our result on the monodromy group we  recall that monodromy groups of \hk manifolds of the deformation types of Beauville's examples were already known: the monodromy group  of manifolds of \kntiposp was computed by Markman \cite[Section 9]{mark_tor} and that of generalized Kummers by the first named author \cite{mon_mon} using fundamental results of Markman \cite{mark_mon}. On the other hand, the monodromy group of 
 manifolds of $OG10$ type (the deformation type of the ten dimensional O'Grady example) was still unknown when this paper first appeared, and has been recently computed by Onorati \cite{ono}.

Our two main applications follow directly from the computation of the monodromy group. 

By the work of  the first named author \cite{mon_kahl} and Markman
 \cite[Section 6]{mark_tor}, the K\"ahler and the Birational K\"ahler cones are cut out by wall divisors and stably prime exceptional divisors respectively. Using that these are two classes of monodromy invariant divisors, we  determine them by means of  explicit geometric constructions in specific examples of \hk manifolds of $OG6$ type in Proposition \ref{prop:wall}.

The key observation in the proof of Corollary \ref{cor:intro_lagr} is that primitive isotropic elements in $H^2(X,\ZZ)$ are in the same monodromy orbit. By fundamental results due to Markman (\cite[Section 5.3 and Lemma 5.17(ii)]{mar_prime}) and Matsushita (\cite[Theorem 1.2]{matsu_lagr}) this reduces the proof of the Hassett-Tschinkel, Hybrechts and Sawon conjecture for \hk manifolds of OG6 type to the existence of a lagrangian fibration in a single case.

The paper is structured as follows: in Section \ref{sec:prelim}, we introduce some basic facts about lattices and monodromy which will be used in the proofs. In Section \ref{sec:abelian}, we prove that parallel transport of complex tori of dimension two can all be obtained by considering families of projective tori. We believe this result is known to experts, but we could not find it in the literature. In Section \ref{sec:sing_mon}, we study the monodromy group of the Albanese fibre of a singular moduli spaces of semistable sheaves, using parallel transport along projective families and Fourier-Mukai transformations. The key result of this section is that this group is already maximal, see Proposition \ref{propsection} for details. In Section \ref{sec:monotutta}, we use the previous result to finish the computation of the Monodromy group. Finally, the last two sections are dedicated to applications: in Section \ref{sec:mov}, we give the structure of the K\"ahler and the birational K\"ahler cones for manifolds of $OG6$ type and in Section \ref{sec:lagr} we prove that the existence of divisors of square zero implies the existence of a birational lagrangian fibration with fibres of polarization type (1,2,2) and we prove Beauville's weak splitting property for these manifolds. 

\section*{Acknowledgements}
We are grateful to Francesco Esposito, Claudio Onorati and Giulia Saccà for useful discussions and to Kieran G. O'Grady for useful suggestions. This work, even when not stated explicitly, uses beautiful and fundamental results of Eyal Markman on monodromy groups, and we would like to acknowledge the importance of his ideas for our field in general and our work in particular. We would also like to thank the Referee for his insightful comments.\\ GM is supported by ``Progetto di ricerca INdAM per giovani ricercatori: Pursuit of IHS''. GM is member of the INDAM-GNSAGA and received support from it. AR acknowledges the MIUR Excellence Department Project awarded to the Department of Mathematics, University of Rome Tor Vergata, CUP E83C1800010000. Part of the results of this paper were proven or written during the second and fourth editions of the Japanese-European symposium on symplectic varieties and moduli spaces.

\section{Preliminaries}\label{sec:prelim} 
In this section, we will gather notation, definitions and some results concerning Lattices and Monodromy groups, which we will use in the following.
\subsection{Lattices: Notation and basic results} \label{ssec:lattice}
In this subsection, we will fix some notation concerning lattices and will recall the results which we will use in the following.
\begin{defn}\label{def1}
An even  lattice $L$ is a finitely generated  free $\mathbb{Z}$ module equipped with a non degenerate bilinear symmetric form $(\cdot,\cdot)$, with values in $\ZZ$, such that the associated quadratic form  takes only even values.

The discriminant group of $L$ is the finite abelian group $A_L:=L^\vee/L$  
and the discriminant form $q_{A_L}:A_L\rightarrow \mathbb{Q}/2\mathbb{Z}$ is the quadratic form induced by the bilinear form $(\cdot,\cdot)$ on $L$.

The divisibility $div_{L}(v)$ of an element $v\in L$ is  the positive generator of the ideal $(v,L)$, if no confusion can arise we simply denote it by $div(v)$.

\end{defn}
We are interested in certain subgroups of the group of isometries $O(L)$ of the even lattice $L$, when it is not negative definite.
To introduce them, we recall that the Grassmannian $\mathbb{G}r^+(L)$ parametrizing maximal positive definite subspaces of $L\otimes \mathbb{R}$
is contractible and the Grassmannian $\mathbb{G}r^{+,or}(L)$ parametrizing oriented maximal positive definite subspaces of $L\otimes \mathbb{R}$
has two connected components.

The  subgroups of  $O(L)$ we are interested in, are the following:
\begin{itemize}
\item $SO(L)$, the group of isometries of determinant one.
\item $O^+(L)$, the group of isometries acting trivially on the set of connected components of $\mathbb{G}r^{+,or}(L)$.
\item $SO^+(L):=SO(L)\cap O^+(L)$.
\item $\widetilde{O}(L)$ The group of isometries whose induced action on $A_L$ is trivial.
\item $S\widetilde{O}^+(L)=SO^+(L)\cap \widetilde{O}(L)$.
\end{itemize}

\begin{oss}\label{positive cone}
For  every even lattice $L$ that is not negative definite the positive cone of $L$ is   
$$C(L):=\left\{ v\in L\otimes \mathbb{R}|\; (v,v)>0 \right\}\subseteq L\otimes \mathbb{R}.$$ 
As shown in \cite[\S 4]{mark_tor}, for every maximal positive subspace $W\subseteq L\otimes \mathbb{R}$, 
the complement of the origin $W\setminus \left\{0\right\}$ is a deformation retract of $C(L)$: hence $C(L)$
has the Homotopy type of a sphere. The subgroup $O^+(L)$ is equivalently defined, in \cite[\S 4]{mark_tor}, as the subgroup of $O(L)$ acting trivially on the homology of $C(L)$, i.e. the subgroup of $O(L)$ preserving the orientation of the positive cone $C(L)$.
\end{oss}
All results of this subsection are well known to experts, and we use standard references for them.

Lemma \ref{lemma:estendi} is a folklore result, which allows to extend isometries of a sublattice to the ambient lattice. We will use it several times. As a preliminary step, we recall the following Lemma \ref{overlattice} characterizing finite index extensions of even lattices.
. 
\begin{lem}\label{overlattice}\cite[Prop. 1.4.1]{Nik80}
Let $\Lambda,N$ be two even lattices such that $\Lambda\subset N$ is a finite index extension. Then $N$ is naturally a subgroup of $\Lambda^\vee$ and $H_N:=N/\Lambda\subset A_\Lambda$ is an isotropic subgroup. Conversely, for every isotropic subgroup $H\subset A_\Lambda$ there exists a finite index extension $\Lambda\subset N$, unique up to isomorphism, such that $H=H_N$. 


\end{lem}
\begin{proof}
The inclusion  $N\subseteq \Lambda^\vee$ is given by the group morphism sending   an element $n\in N$  to the morphism of $\phi_{n}\in\Lambda^\vee$ such that $\phi_{n}(v)=(n,v)$ for $v\in \Lambda$.
The value of the discriminant form on $[n]$ is the value of the quadratic form on $n$ modulo $2\mathbb{Z}$, that is $0$ as $N$ is even. 
Conversely the even lattice $N$ can be obtained as the inverse image of $H$ in $\Lambda^\vee$.  
\end{proof}
\begin{lem}\label{lemma:estendi}
Let $L$ and $M$ be two lattices and let $L\oplus M\subset N$ be a finite order extension by an even lattice. 
Let $f\in \widetilde{O}(L)$ be an isometry acting trivially on $A_L$. 
Then, there exists an isometry $\overline{f}$ of $N$ such that $\overline{f}_{|L}=f$ and $\overline{f}_{|M}=Id_M$.
\end{lem}
\begin{proof}
Since the order of the extension is finite, for all $n\in N\subseteq  L^{\vee} \oplus M^{\vee}$ there are $l\in L$, $m\in M$ and  an integer $r$ such that $rn=l+m$ or, equivalently, $n=l/r+m/r$ in $L^{\vee} \oplus M^{\vee}$.

Since $f\in \widetilde{O}(L)$, we have $[f(l)/r]=[l/r]$ in $A_L$: hence there exists $l_{0}\in L$ such that $f(l)/r=l/r+l_0$ in $L^{\vee}$. 
It follows that $f(l)/r+m/r\in N\subseteq  L^{\vee} \oplus M^{\vee}$ and, setting $\overline{f}(n):= f(l)/r+m/r$ for all $n\in N$, we get the desired isometry. 
\end{proof}
\begin{oss}\label{est2}
The above lemma will be often applied to the case where $f\in \widetilde{O}(L)$ is an isometry of a sublattice $L$ of $U^k$ and  $M$ has rank one:
if $M$ is generated by an element of square $2$ and $M$ is its perpendicular, the discriminant group 
$A_M$ has order two and the extension $\overline{f}\in O(U^k)$ always exists.
\end{oss} 

The following 
famous result of Eichler
allows to determine when two elements of a lattice are in the same orbit of the isometry group.
\begin{lem}\cite[Lemma 3.5]{GHS10}\label{lem:eichler}
Let $L'$ be an even lattice and let $L=U^2\oplus L'$. Let $v,w\in L$ be two primitive elements such that the following holds:
\begin{itemize}\renewcommand{\labelitemi}{$\bullet$}
\item $v^2=w^2$.
\item $[v/div(v)]=[w/div(w)]$ in $A_L$.
\end{itemize}
Then there exists an isometry $g\in S\widetilde{O}^+(L)$ such that $g(v)=g(w)$.
\end{lem}
In the reference above, the result is stated for $g\in \widetilde{O}^+(L)$, but the proof uses a class of isometries, called Eichler's transvections, which have determinant one. Indeed, let $L$ be any indefinite lattice, let $e\in L$ be isotropic and let $a\in e^\perp$, the Eichler's transvection $t(e,a)$ with respect to $e$ and $a$ is defined as follows:
\begin{equation}\label{eq:eichler}
t(e,a)(v)=v-(a,v)e+(e,v)a-\frac{1}{2}(a,a)(e,v)e.
\end{equation}
In particular, we have the following three properties:
\begin{align}\label{eq:eichler_add1}
t(e,a)^{-1}&=t(e,-a)\\\label{eq:eichler_add2}
t(e,a)\circ t(e,b)&=t(e,a+b).\\\label{eq:eichler_add3}
g^{-1}\circ t(e,a)\circ g&=t(g(e),g(a)), 
\end{align}
where $g$ is any isometry.
If $L=U\oplus L_1$, we will denote with $E_U(L_1)$ the group of all Eichler's transvections $t(e,a)$ with $e\in U$ and $a\in L_1$.

Finally we recall two useful results giving a finite number of generators of isometry groups and concerning Eichler's transvections. 
\begin{lem}\label{lem:minieichler}\cite[Lemma 3.2]{GHS09}
Let $U^2$ be two copies of the hyperbolic lattice with standard basis $\{e_1,f_1,e_2,f_2\}$. Then $SO^+(U^2)$ is generated by $t(e_2,e_1),t(e_2,f_1),t(f_2,e_1)$ and $t(f_2,f_1)$.
\end{lem}
\begin{lem}\label{lem:supereichler}\cite[Prop. 3.3 (iii)]{GHS09}
Let $L:=U\oplus L_1$ be an even lattice. Then $O^+(L)$ is generated by $O^+(L_1)$ and $E_U(L_1)$.
\end{lem}


\subsection{Monodromy: definitions and facts}\label{ssec:monodromy}

In this subsection we recall the notions of parallel transport operator and monodromy groups.
\begin{defn}
A proper morphism $f\,:\,\mathcal{X}\rightarrow T$  between complex analytic spaces, is a proper  analytically  locally (on $\mathcal{X}$) trivial family
if every $x\in\mathcal{X}$ has an analytic neighborhood  isomorphic, over $T$, to a product of a neighborhood of $U_x$ of $x$ in its fibre $f^{-1}(f(x))$ and a neighborhood $V_{f(x)}$ of $f(x)\in T$.
\end{defn}
The main example of a proper analytically locally trivial family is given by a  proper smooth morphism between complex manifolds.
A proper  analytically locally trivial family $f\,:\,\mathcal{X}\rightarrow T$ is, locally on $T$, topologically trivial: 
hence for every $n\in\mathbb{N}$, the higher direct image  $R^nf_{*}(\ZZ)$ is locally constant.
In particular for every path $\gamma:[0,1]\rightarrow T$, the sheaf $\gamma^{*}(R^nf_{*}(\ZZ))$ is constant.
\begin{defn}\label{trasppar}
\begin{enumerate}
\item{Set $X:=f^{-1}(\gamma(0))$ and $X':=f^{-1}(\gamma(1))$, the parallel transport operator on $H^{n}$
associated with $f$ and $\gamma$ is the isomorphism
$$t^{n}_{\gamma,f}: H^n(X,\ZZ)\rightarrow  H^n(X',\ZZ)$$
induced between the stalks at $0$ and $1$ by the trivialization of $\gamma^{*}(R^nf_{*}(\ZZ))$.}
\item{A monodromy operator on $H^n(X,\ZZ)$ induced by $f$ is an isomorphism of the form  
$$t^{n}_{\gamma,f}: H^n(X,\ZZ)\rightarrow  H^n(X,\ZZ)$$ where  $\gamma$ is a loop ($\gamma(0)=\gamma(1)$).}
\item{The group of monodromy operators on $H^n(X,\ZZ)$ induced by $f$ is
$$Mon^{n}_{f}(X):=\left\{ t^{n}_{\gamma,f}|\; \gamma(0)=\gamma(1))\right\}.$$} 
\end{enumerate}\end{defn} 
By construction the parallel transport operator $t^{n}_{\gamma,f}$ only depends on the fixed endpoints homotopy class of $\gamma.$
The notions of parallel transport operator and monodromy operator   allow to introduce the monodromy  groups that we are intersted in.
\begin{defn}
\begin{enumerate}
\item{If $X$ is a compact K\"ahler manifold, the monodromy group $Mon^n(X)$ is the subgroup of $\Aut_{\ZZ-mod}(H^n(X,\ZZ))$
generated by the subgroups of the form $Mon^{n}_{f}(X)$ where $f:\mathcal{X}\rightarrow T$ is a proper smooth morphism whose fibres 
are compact K\"ahler manifolds and $X$ is a fiber of $f$.}
\item{If $X$ is a projective manifold, the projective monodromy group  $Mon^n(X)^{pr}$ is the subgroup of $\Aut_{\ZZ-mod}(H^n(X,\ZZ))$
generated by the subgroups of the form $Mon^{n}_{f}(X)$ where $f:\mathcal{X}\rightarrow T$ is a projective smooth morphism
between algebraic varieties 
having $X$ as a fibre.}
\item{If $X$ is a singular projective variety, the locally trivial projective monodromy group $Mon^n(X)^{pr}_{lt}$ is the subgroup of $\Aut_{\ZZ-mod}(H^n(X,\ZZ))$
generated by the subgroups of the form $Mon^{n}_{f}(X)$ where $f:\mathcal{X}\rightarrow T$ is a proper analytically locally trivial family whose fibers are projective varieties and $X$ is a fiber of $f$.}
\end{enumerate}\end{defn}

We will be interested in the special case of the group $Mon^2(X)$, where either $X$ is a \hk manifold
or $X$ is a projective variety admitting a resolution by a \hk manifold: in the latter, $X$ is a singular simplectic variety 
such that  $H^1(\mathcal{O}_{X})=0$ admitting a unique, up to scalar, holomorphic two form on the smooth locus, i.e. $X$ is a projective primitive symplectic variety according to \cite[Definition 3.1]{bakkerlehn} or a projective Namikawa symplectic variety according to \cite[Definition 2.18]{perult}.
In both cases, $H^2(X,\ZZ)$ has  a pure Hodge structure with a compatible deformation invariant quadratic form $B_X$, the Beauville-Bogomolov form in the smooth case and the Beauville-Bogomolov-Namikawa form in the singular case (see \cite{nami_sing}).
The deformation invariance of $B_X$ implies that, $Mon^{2}(X)$ in the smooth case or in $Mon^{2}(X)^{pr}_{lt}$ in the singular case,
actually lies in  $O(H^2(X,\ZZ))$, where the lattice structure is given by $B_X$.

For a \hk manifold $X$ or  a projective primitive symplectic variety, the group $O^+(H^2(X,\ZZ))$ coincides with the group of isometries preserving the two components of the cone of the classes in $H^{1,1}(X,\mathbb{R})$ having strictly positive Beauvill-Bogomolov(-Namikawa) square (see \cite[\S 4]{mark_tor}). As a K\"ahler class in the smooth case, or an ample class in the singular projective case, gives a preferred component of this cone, we have the following fundamental constraint on the monodromy groups that we are going to study:
\begin{lem}\label{o+}
\begin{enumerate}\item{ $Mon^2(X)\subseteq O^+(H^2(X,\ZZ))$ for every  \hk manifold $X$.}
\item{$Mon^2(X)^{pr}_{lt}\subseteq O^+(H^2(X,\ZZ))$ for every projective primitive symplectic variety $X$.}
\end{enumerate}
\end{lem}

\section{Monodromy of Abelian surfaces}\label{sec:abelian}
In this section we study the monodromy group $Mon^2(A)$ on the second cohomology for an
abelian surface $A$. The elements of $Mon^2(A)$ will be used in the next section to induce monodromy operators on the Albanese fibres of moduli spaces of sheaves on abelian surfaces.
It is well known, essentially already contained in \cite{shioda}, that  $$Mon^2(A)=SO^+(H^{2}(A,\mathbb{Z})),$$ where the lattice structure on  
$H^{2}(A,\mathbb{Z})$ is given by the intersection form.
We recall that $H^{2}(A,\mathbb{Z})= \Lambda^2 H^1(A,\mathbb{Z})\cong U^3$ and  we have an isomorphism $$Mon^2(A)\simeq SO^+(U^{3}).$$

Unfortunately we cannot use directly this result in the next section to induce monodromy operators on Albanese fibres of moduli spaces of sheaves 
because of the absence of a satisfactory theory of moduli spaces of sheaves on non necessarily  projective surfaces.
  
We need to prove a more precise result stating that the whole  $Mon^2(A)$
comes from compositions of monodromy operators induced by projective families of polarized abelian surfaces. 
A polarized abelian surface of degree $2d\in\mathbb{N}\setminus\left\{0\right\}$ is a pair $(A,h)$, where $A$
is an abelian surface and $h\in H^{2}(A,\ZZ)\cap H^{1,1}(A)$ is an indivisible  class represented by an ample line bundle $H$
on $A$ such that $H^{2}=2d$.
 
In the following Proposition \ref{prop:piatesko1} and Corollary \ref{cor:piatesko2}, for every $d\in\mathbb{N}\setminus\left\{0\right\}$ we analyze monodromy operators coming from a specific  family 
\begin{equation}\label{intro}
f_{2d}:\mathcal{A}_{2d}\rightarrow T_{2d}
\end{equation}  
of polarized abelian surfaces of degree $2d$  obtained  
as follows.

Since for every  ample line bundle $H$ of degree $2d$ on an abelian surface $A$ the line bundle $3H$ is always very ample \cite[Theorem 4.5.1]{Birkenhake-Lange} and, by Riemann-Roch and Kodaira vanishing, the linear system $|3H|$ has dimension $m=9d-1$,
the  polarized abelian surfaces with a primitive polarization of degree $2d$  can all be embedded in a fixed projective space 
$\mathbb{P}^{m}$.

It follows that there exists a Zariski  open subset $T_{2d}$ 
of the  Hilbert scheme parametrizing subschemes of $\mathbb{P}^{m}$ with Hilbert polynomial $P(x)=9dx^2$ 
such that, for every $t\in T_{2d}$, the corresponding subscheme $A_{t}\subset \mathbb{P}^{m}$ is an abelian surface  equipped with an ample line bundle $H_t$ whose cohomology class $h_t\in H^{2}(A_t,\mathbb{Z})$ is indivisible and such that the complete linear system $|3H_t|$ is the linear system of the  hyperplane sections on  $A_t$. Moreover every polarized abelian surface of degree $2d$ is isomorphic to
$(A_t,h_t)$ for some $t\in T_{2d}$. 
By definition 
$f_{2d}:\mathcal{A}_{2d}\rightarrow T_{2d}$ is the base change over $T_{2d}$ of the universal family of the previous Hilbert scheme.

\begin{oss}\label{T2D} The scheme  $T_{2d}$ is a  smooth connected variety of dimension $(m+1)^{2}+2$. 

In fact, since the moduli space of polarized abelian surfaces has dimension $3$ and $PGL(m+1)$ acts on $T_{2d}$ with finite stabilizers, for every $t\in T_{2d}$ the dimension of $T_{2d}$ at $t$ is at least  $(m+1)^{2}+2$. On the other hand, the tangent space of $T_{2d}$ at $t$ is the space of sections of the normal bundle 
$N_t$  to $A_t$ in $\mathbb{P}^{m}$ and, letting $T_{A_{t}}$ be the tangent bundle of $A_t$ and  $T_{\mathbb{P}^{m}|A_{t}}$ be the restriction on $A_t$ of the tangent bundle of $\mathbb{P}^m$, we have the exact sequence 
$$0\rightarrow H^{0}(T_{A_{t}})\rightarrow H^{0}(T_{\mathbb{P}^{m}|A_{t}})
\rightarrow H^{0}(N)\rightarrow H^{1}(T_{A_{t}}).$$
Since the image of the last map is contained in the  codimension one  subspace of first order deformations of $A_{t}$ where $h_t$ stays algebraic, we deduce the inequality
$h^{0}(N)\le h^{1}(T_{A_{t}})-1+h^{0}(T_{\mathbb{P}^{m}|A_{t}})- h^{0}(T_{A_{t}})$.
Since $H^i(\mathcal{O}_{\mathbb{P}^{m}}(1)_{|A_{t}})\simeq H^i(\mathcal{O}_{\mathbb{P}^{m}}(1)$ for $i=1,2$
the long exact sequence in cohomology associated with the restriction of the Euler exact sequence of $\mathbb{P}^m$ to $A_t$ gives 
$h^{0}(T_{\mathbb{P}^{m}|A_{t}})=h^{0}(T_{\mathbb{P}^{m}})+ h^{1}(\mathcal{O}_{A_{t}})$, hence 
$$h^{0}(N)\le h^{1}(T_{A_{t}})-1+h^{0}(T_{\mathbb{P}^{m}})+ h^{1}(\mathcal{O}_{A_{t}})- h^{0}(T_{A_{t}})= (m+1)^{2}+2$$
and $T_{2d}$ is a smooth variety of dimension $(m+1)^{2}+2$.

Connectedness holds because the  moduli space of polarized abelian surfaces of degree $2d$ is connected and points in 
$T_{2d}$ representing isomorphic polarized abelian varieties of degree $2d$ belong to the same $PGL(m+1)$ orbit.
\end{oss}  

Since $f_{2d}:\mathcal{A}_{2d}\rightarrow T_{2d}$ is a family of complex tori its monodromy is completely determined 
by its action on the first cohomology group and we first analyze this action. 
Recall that a polarized abelian surface $(A,h)$ of degree $2d$ is naturally equipped with
an integral, bilinear, non degenerate alternating form $$F_{h}:H^{1}(A,\mathbb{Z})\times H^{1}(A,\mathbb{Z})\rightarrow \mathbb{Z}$$ defined by $$F_{h}(\beta_1,\beta_2)=\int_{A}h\wedge\beta_1\wedge\beta_2$$ for $\beta_{i}\in H^{1}(A,\mathbb{Z})$. 
Naturality of $F_{h}$ implies that monodromy operators on $H^{1}(A,\mathbb{Z})$ associated with  polarized deformations of  abelian surfaces of degree $2d$ always belong to the group 
$$Sp(H^1(A,\ZZ),F_{h}):=$$$$\left\{\psi\in GL(H^1(A,\ZZ)): F_{h}(\psi(\beta_1),\psi(\beta_2))=F_{h}(\beta_1,\beta_2),\; \forall \beta_i\in 
H^{1}(A,\ZZ) \right\}$$ of the automorphisms of $H^1(A,\ZZ)$ preserving $F_h$.
The following proposition shows that the monodromy group of the polarized family of abelian surfaces $f_{2d}$ is as big as possible. 

\begin{prop}\label{prop:piatesko1}
Suppose that  $0$ is a point of $T_{2d}$ and set  $A:=A_0$  and $h:=h_0$. Then
$$Mon^{1}_{f_{2d}}(A)=Sp(H^1(A,\ZZ),F_{h}).$$
\end{prop}
\begin{proof}
Since $f_{2d}:\mathcal{A}_{2d}\rightarrow T_{2d}$ is the restriction of the universal family of a Hilbert scheme
parametrizing subschemes of $\mathbb{P}^m$, the line bundle $\mathcal{O}_{\mathbb{P}^m}(1)$ induces a line bundle $\mathcal{L}$ on 
$\mathcal{A}_{2d}$ and, by construction, the cohomology class of the restriction $[\mathcal{L}_{|A}]$ of $\mathcal{L}$ to $A$ is $3h\in H^2(A,\ZZ)$. It follows that $h$ is invariant under the action of 
$Mon^{2}_{f_{2d}}(A)$: hence $Mon^{1}_{f_{2d}}(A)\subseteq Sp(H^1(A,\ZZ),F_{h})$.

To prove the opposite inclusion we consider the monodromy action of the 
fundamental group $\pi_{1}(T_{2d},0)$  on $Sp(H^1(A,\ZZ),F_{h})$ by composition on the left and  
let $e:T'\rightarrow T_{2d}$ be the \'etale cover of the connected complex manifold $T_{2d}$ 
induced by this right action. 
More explicitly,  $T'$ is in natural bijective correspondence with the set of pairs $(t,\zeta)$ where $t\in T_{2d}$ and $\zeta:  H^{1}(A,\mathbb{Z})\rightarrow H^{1}(A_t,\mathbb{Z})$
is an isomorphism sending   the alternating bilinear form 
$F_{h_{t}}:H^{1}(A_t,\mathbb{Z})\times H^{1}(A_t,\mathbb{Z})\rightarrow \mathbb{Z}$  associated with the polarized abelian surface 
$(A_t,h_t)$ to $F_h$. 
 
By construction, the inclusion $ Sp(H^1(A,\ZZ),F_{h})\subseteq Mon^{1}_{f_{2d}}(A)$ holds if and only if $T'$ is connected.

Connectedness of $T'$ will be shown by analyzing the relevant period map for the corresponding family of abelian surfaces.
In order to define the desired period map we recall that, since $(A,h)$ is a polarized abelian surface of degree $2d$,  
there exists an isomorphism $\iota:\mathbb{Z}^{4} \rightarrow H^{1}(A,\mathbb{Z})$ sending  $F_h$ to 
the integral, bilinear, non degenerate, alternating form $F:\mathbb{Z}^{4}\times \mathbb{Z}^{4}\rightarrow \mathbb{Z}$ whose associated matrix is  
$$\begin{vmatrix}
0&0&1&0\\
0&0&0&d\\
-1&0&0&0\\
0&-d&0&0\\
\end{vmatrix}.$$

Let $\mathbb{G}r^{F}(2,4)$ be the lagrangian Grassmannian parametrizing 
$2$-dimensional complex vector spaces in $\mathbb{C}^4$ that are isotropic with respect to the $\mathbb{C}$-linear extension $F_{\mathbb{C}}$ of $F$ and set 
$$D:=\left\{ V\in \mathbb{G}r^{F}(2,4):\; V\cap\overline{V}=0\;\&\; 
iF_{\mathbb{C}}(v,\overline{v})>0\; \forall v\in V \right\}.$$
By Hodge theory (see subsection  7.1.2 of \cite{VoisinI}) the subspace $H^{1,0}(A_{t})\subset H^{2}(A_{t},\mathbb{C})$
is lagrangian  with respect to the $\mathbb{C}$-linear extension $F_{h_t,\mathbb{C}}$ of $F_{h_t}$ and positive definite with respect to the hermitian form associated with $F_{h_t}$.  Since $\zeta\circ\iota$ sends $F_{h_t}$ to $F$,  the subspace  
$(\zeta\circ\iota)^{-1}(H^{1,0}(A_{t}))$ belongs to $D$ for every $t\in T_{2d}$. Hence there exists a 
holomorphic period map $P:T'\rightarrow D$  defined by $$P(t,\zeta):=(\zeta\circ\iota)^{-1}(H^{1,0}(A_{t})).$$ 

Since every $V\in D$ defines a polarized abelian surface of degree $2d$ whose underlying complex torus  $A_{V}:=\frac{\mathbb{C}^{4\vee}}{\overline{V}^{\vee}+\mathbb{Z}^{4\vee}}$ is  equipped with an identification $H^{1}(A_{V},\mathbb{Z})\simeq \mathbb{Z}^4$ sending $H^{1,0}$ to $V$, the holomorphic map $P$ is surjective.

Moreover the fibres of $P$ are irreducible  of  dimension $(m+1)^2-1$. More precisely the  action of the connected group $PGL(m+1)$ on the manifold $T_{2d}$ lifts to an action on the manifold $T'$: the lifted action is defined by setting, $g(t,\zeta):=(gt,g_*\circ\zeta)$ where  $g_{*}:H^1(A_t,\mathbb{Z})\rightarrow H^1(A_{gt},\mathbb{Z})$ is the isomorphism induced by the restriction of $g\in PGL(m+1)$ to a morphism from
$A_{t}$ to $A_{gt}$.

Since $g_*$ is a morphism of Hodge structures, points of $T'$ in the same $PGL(m+1)$ orbit have the same image in $D$. 
Viceversa, $P(t_1,\zeta_1)=P(t_2,\zeta_2)$ implies that 
$\zeta_2\circ\zeta_{1}^{-1}:H^1(A_{t_{1}},\mathbb{Z})\rightarrow 
H^1(A_{t_{2}},\mathbb{Z})$ is an isomorphism of integral polarized Hodge structures: hence it comes from an isomorphism $A_{t_{2}}\simeq A_{t_{2}}$ sending $h_{t_1}$ to $h_{t_2}$. Composing with translations, we may also assume that this isomorphism sends $\mathcal{O}_{\mathbb{P}^m}(1)_{|A_{t_1}}$ to $\mathcal{O}_{\mathbb{P}^m}(1)_{|A_{t_2}}$, i.e it is the restriction of an element of $PGL(m+1)$. We conclude that the fibres of $P$ are irreducible because they are the orbits of the $PGL(m+1)$-action on $T'$.
The dimension of every fibre is $(m+1)^2-1$, since the group of automorphisms of an abelian surface fixing an ample line bundle is finite,  hence $PGL(m+1)$ acts with finite stabilizers on $T_{2d}$.

Since $T'$ is an  \'etale cover of  $T_{2d}$, by Remark \ref{T2D} it is a complex manifold  of pure dimension $(m+1)^2+2$ and, since $P:T'\rightarrow D$ is surjective with connected equidimensional fibres, $T'$ is connected if $D$ is connected. Finally the  period domain $D$ is connected because it is isomorphic to the Siegel upper half space $\mathbb{H}_2$ parametrizing $2\times 2$ complex symmetric matrices with positive definite imaginary part.
\end{proof}

\begin{oss} The Plucker embedding identifies $D$ with one of the two components of the period domain $\Delta_{\alpha}:=\left\{[l]\in \mathbb{P}(U^3\otimes \mathbb{C}):\; l\cdot \alpha =0,\;
l\cdot l=0 \;\&\; l\cdot \overline{l}>0 \right\}$ for the Hodge structure on $H^{2}(A,\ZZ)$.
\end{oss}

As a consequence we obtain the group of monodromy operators of $f_{2d}$ on second cohomology groups.
In order to state the result, for every polarized abelian surface $(A,h)$ of degree $2d$ set
$$SO^+_{h}(H^{2}(A,\ZZ)):=\{\varphi\in SO^+(H^{2}(A,\ZZ))\,|\,\varphi(h)=h\}.$$

\begin{cor}\label{cor:piatesko2}
Suppose that  $0$ is a point of $T_{2d}$ and set  $A:=A_0$  and $h:=h_0$. Then
$$Mon^{2}_{f_{2d}}(A)=SO^+_{h}(H^{2}(A,\ZZ)).$$
\end{cor}

\begin{proof} For every free $\mathbb{Z}$-module $M$, let $GL(M)$ be the automorphism group of $M$.
Since $H^2(A,\mathbb{Z})$ and $\Lambda^2(H^{1}(A,\mathbb{Z}))$ are isomorphic as modules under the monodromy actions of $\pi_{1}(T_{2d},0)$
there is a group morphism $$\lambda:GL(H^1(A,\ZZ))\rightarrow GL(H^{2}(A,\ZZ))$$ 
sending an automorphism $\psi$ of $H^1(A,\ZZ)$ to the automorphism $\Lambda^{2}(\psi)$ of $H^{2}(A,\ZZ)$
and, by Proposition \ref{prop:piatesko1}, it suffices to show that 
$$\lambda(Sp(H^1(A,\ZZ),F_{h}))=SO^+_{h}(H^{2}(A,\ZZ)).$$

We first remark  that $\lambda(SL(H^1(A,\ZZ)))=SO^+(H^{2}(A,\ZZ))$. By Lemme 4 of \cite[EXPOS\'E VIII]{beau_ast}, 
an isometry $\phi\in O(H^{2}(A,\ZZ))$ belongs to $SO(H^{2}(A,\ZZ))$ if and only if either $\phi$ or $-\phi$ belongs to the image of 
$\lambda$. Since for $\phi\in GL(H^1(A,\ZZ))$ the automorphisms $\phi$ and $\lambda(\psi)$ have the same determinant we deduce the inclusion 
$\lambda(SL(H^1(A,\ZZ)))\subset SO(H^{2}(A,\ZZ))$ and the index of this inclusion is $2$. By connectedness of 
$SL(H^1(A,\mathbb{R}))$ the subgroup  $\lambda(SL(H^1(A,\ZZ)))$ is contained in the index two subgroup  $SO^+(H^{2}(A,\ZZ))$
acting trivially on the set of the connected components of $\mathbb{G}r^{+,or}(H^{2}(A,\ZZ))$: hence $\lambda(SL(H^1(A,\ZZ)))=SO^+(H^{2}(A,\ZZ))$.

To finish the proof we notice that, by definition of $F_h$, an automorphism $\psi\in SL(H^1(A,\ZZ))$ preserves the form $F_h$ if and only if   $\lambda(\psi)$ preserves $h$. By the last equality we deduce that, for every $\phi \in SO^+_{h}(H^{2}(A,\ZZ))$ there exists $\psi\in SL(H^1(A,\ZZ))$ such that $\lambda(\psi)=\phi$ and, since $\phi(h)=h$, we obtain that $\psi\in Sp(H^1(A,\ZZ),F_{h})$: hence 
$SO^+_{h}(H^{2}(A,\ZZ))\subseteq \lambda(Sp(H^1(A,\ZZ),F_{h}))$.   
The opposite inclusion holds because, since  $F_h$ is symplectic,  $Sp(H^1(A,\ZZ),F_{h})\subset SL(H^1(A,\ZZ))$ and, as above,
$\lambda(\psi)(h)=h$ for every $\psi\in Sp(H^1(A,\ZZ),F_{h})$.
\end{proof}

Using an isometry $H^{2}(A,\mathbb{Z})\simeq U^{3}$ and letting 
$a\in U^{3}$ be the image of $h$ the equality of \ref{prop:piatesko1}
can be restated as 
$$Mon^{2}_{f_{2d}}(A_t)\simeq
SO^{+}_a(U^3):=\left\{\varphi\in SO^{+}_a(U^3)|\; \varphi(a)=a \right\}.$$

The following, purely lattice theoretic Lemma shows that a few subgroups of the form $SO^{+}_a(U^3)$ are sufficient to generate the whole
$SO^{+}(U^3)$.
\begin{lem} \label{latticiancora}Let $\left\{e_1,f_1\right\}$ be a standard basis for a copy $U_1$ of $U$ in $U^3$. For every $d\in \mathbb{N}\setminus\left\{0\right\}$, 
the union 
\begin{equation} \label{unione}SO^{+}_{e_1+df_1}(U^3)\cup SO^{+}_{de_1+f_1}(U^3)\cup SO^{+}_{e_1+(d+1)f_1}(U^3)\cup SO^{+}_{(d+1)e_1+f_1}(U^3)
\end{equation} 
in $SO^{+}(U^3)$
generates the whole group. 
\end{lem}
\begin{proof}

Let $g$ be any isometry of $SO^+(U^{3})$. Let $L_1\cong U^2$ be the orthogonal of $U_1$ inside $U^3$ and let $\{ e_2,f_2,e_3,f_3\}$ be the standard basis for $L_1$. By Lemma \ref{lem:supereichler}, every element of $O^+(U^{3})$ is a composition of elements in $O^+(L_1)$ and transvections of the form $t(e_1,a)$ or $t(f_1,a)$ with $a\in L_1$. As $g$ is in $SO^+(U^3)$, we can write it as a compositions of elements in $SO^+(L_1)$ and transvections as above. Therefore, to prove our claim it is enough to prove that all these factors can be obtained by compositions of elements contained in the four subgroups of equation (\ref{unione}).

Since $SO^+(L_1)$ acts trivially on $U_1$, it is contained in each of the four subgroups. 
We will now prove that $t(e_1,a)$ and $t(f_1,a)$ are contained in the subgroup generated by the union (\ref{unione}) simultaneously. 
By the previous step of the proof, by Lemma \ref{lem:eichler} applied to the lattice $L_1$ and by equation \eqref{eq:eichler_add3}, we can suppose that $a$ lies in the second copy of $U$ spanned by $e_2,f_2$. In particular, the two isometries $t(e_1,a)$ and $t(f_1,a)$ act trivially on the third copy of $U$ and can therefore be considered as an isometry of $\langle e_1,f_1,e_2,f_2\rangle$. By Lemma \ref{lem:minieichler}, $t(e_1,a)$ and $t(f_1,a)$ can be written as a composition of $t(e_2,e_1),t(e_2,f_1),t(f_2,e_1)$ and $t(f_2,f_1)$. 
For every positive integer $d$, the transvections $t(e_2,e_1-df_1)$ and $t(f_2,e_1-df_1)$ fix  $e_1+df_1$: therefore they are elements of $SO^+_{e_1+df_1}(U^3)$.
Analogously, the transvections $t(e_2,e_1-(d+1)f_1)$ and $t(f_2,e_1-(d+1)f_1)$ fix  $e_1+(d+1)f_1$, therefore they are elements of $SO^+_{e_1+(d+1)f_1}(U^3)$.

A direct computation using \eqref{eq:eichler_add1} and \eqref{eq:eichler_add2} shows 
\begin{align*}
t(e_2,e_1-df_1)\circ t(e_2,e_1-(d+1)f_1)^{-1}&=t(e_2,f_1)\\
t(f_2,e_1-df_1)\circ t(f_2,e_1-(d+1)f_1)^{-1}&=t(f_2,f_1)
\end{align*}
Analogously, we obtain $t(e_2,e_1)$ and $t(f_2,e_1)$ by fixing the polarizations $de_1+f_1$ and $(d+1)e_1+f_1$, hence our claim.

\end{proof}
 
The following corollary is an immediate geometric consequence of Lemma \ref{latticiancora} and Corollary \ref{cor:piatesko2}.

\begin{cor}\label{abelian_tuttoproj}
Let  $A$ be the product of two elliptic curves $E_1$ and $E_2$. 
Let $e,f$ be the classes of the two elliptic curves inside $NS(A)$.
For every $d\in \mathbb{N}\setminus\left\{0\right\}$
$$SO^{+}(H^{2}(A,\ZZ))=\langle \bigcup_{k=d}^{d+1}SO^{+}_{ke+f}(H^2(A,\mathbb{Z}))\cup \bigcup_{k=d}^{d+1} SO^{+}_{e+kf}(H^2(A,\mathbb{Z}))\rangle\subseteq 
Mon^2(A)^{pr}.$$
\end{cor} 
\begin{oss}\label{dual}
The most natural example of an isometry in $O^{+}(H^{2}(A,\ZZ))$ that does not belong to $SO^{+}(H^{2}(A,\ZZ))$ (hence nor  to $Mon^{2}(A)$) is provided by
the Poincar\'e duality morphism.

Let $\widehat{A}:=\frac{H^{1}(A,\mathbb{C})}{H^{1}(A,\ZZ)+H^{0,1}(A)}$ be the dual complex torus. There are identifications 
$$H^{n}(\widehat{A},\ZZ)=H^{n}(A,\ZZ)^{\vee},\;H^{p,q}(\widehat{A})=H^{q,p}(A)^{\vee}$$ 
and the Poincar\'e duality morphism $$PD:H^{2}(A,\ZZ)\rightarrow H^{2}(\widehat{A},\ZZ)=H^{2}(A,\ZZ)^{\vee}$$
defined by $$PD(\alpha):=\int_A \alpha \wedge (\cdot)$$
is a isomorphism and a Hodge isometry. 

Since $A$ and $\widehat{A}$ are complex tori the cohomology groups  $H^{1}(A,\ZZ)$ and  $H^{1}(\widehat{A},\ZZ)$
have  canonical orientations inducing orientations on $H^{2}(A,\ZZ)$ and  $H^{2}(\widehat{A},\ZZ)$. As shown in Lemma 3 of \cite{shioda}
the morphism  $PD$ is incompatible with the orientations on 
$H^{2}(A,\ZZ)$ and $H^{2}(\widehat{A},\ZZ)$ and is not the second wedge power of an isomorphism 
between $H^1(A,\ZZ)$ and $H^{1}(\widehat{A},\ZZ)$ (see also \cite[Lemma 4.5]{MarkmanMehrotra}).

If $h\in H^{2}(A,\ZZ)$ is the class of an ample divisor, the same holds for   $\widehat{h}:=PD(h)\in H^{2}(\widehat{A},\ZZ)$  
and $(\widehat{A},\widehat{h})$ is called the dual abelian surface of $(A,h)$.
If $(A,h)$ is a principally polarized abelian surface, i.e. $h^2=2$, there exists an isomorphism of polarized abelian surfaces $g:A\rightarrow\widehat{A}$
and, since isomorphisms between  abelian surfaces are always compatible  with the orientations induced by the complex structures,  $g^{*}\circ PD:H^{2}(A,\ZZ)\rightarrow H^{2}(A,\ZZ)$ is an element of  $O^{+}(H^{2}(A,\ZZ))\setminus SO^{+}(H^{2}(A,\ZZ))$.
\end{oss}

\section{Monodromy of the singular model} \label{sec:sing_mon}
In this section we study locally trivial monodromy of singular symplectic  varieties arising as Albanese fibers of moduli spaces of sheaves on general Abelian surfaces, whose desingularization are manifolds of $OG6$ type. 

In order to properly state the results of this section, we need to fix our setting for this and the next  section. 

Let us first say that, in the whole paper, for every abelian surface $A$ we freely use
the identification $H^{4}(A,\mathbb{Z})=\mathbb{Z}$ provided by the isomorphism sending the Poincar\'e dual $\eta$ of a point on $A$ to $1$.  
\begin{ass} \label{set}Let $A$ be an abelian surface and let $$w=(w_{0},w_2,w_4)\in H^{0}(A,\mathbb{Z})\oplus NS(A)\oplus H^{4}(A,\mathbb{Z})$$ be a Mukai vector such that $w_0>0$ or $w_0=0$, $w_2$ is effective and $w_4\ne 0$ or $w_0=0$, $w_2=0$ and $w_4>0$. Assume that the Mukai square $w^2:=w_2^2-2w_0w_4$ of $w$ is $2$ and set $v=(v_0,v_2,v_4):=(2w_0,2w_2,2w_4)=2w$. 

For a $v$-generic polarization  $H$ on $A$ (see \cite[Definition 2.1]{PerRap}), let $h\in H^{2}(A,\ZZ)$ be its class. 

Let $M_v(A,H)$ be the Gieseker moduli space  of $H$-semistable sheaves on $A$ with Mukai vector $v$. 

Let $K_{v}(A,H)\subset M_v(A)$ be a fibre of the (isotrivial) Albanese fibration of $M_v(A,H)$ and let $\widetilde{K}_{v}(A,H)$ be the the blow up of $K_{v}(A,H)$ along its singular locus with reduced structure. 

By  Theorem 1.6 of \cite{pr_crelle}, the projective variety  $\widetilde{K}_{v}(A,H)$  is a hyperk\"ahler manifold  in the deformation class of $OG6$.

\end{ass}

In this section we determine the group $Mon^2(K_v(A,H))_{lt}^{pr}$ of monodromy operators on $H^{2}(K_v(A,H),\mathbb{Z})$ that are compositions of parallel transport operators along projective families  which are analytically 
locally trivial deformations at every point of the domain. 
The cohomology $H^{2}(K_v(A,H),\mathbb{Z})$ has a lattice structure given by the Beauville-Bogomolov-Namikawa pairing, i.e. the restriction of the Beauville-Bogomolov pairing on the resolution of $K_v(A,H)$. This lattice structure is invariant under deformations which are analitically 
locally trivial deformations at every point of the domain (see \cite[Lemma 5.5]{bakkerlehn}), therefore  $Mon^2(K_v(A,H))_{lt}^{pr}$ is contained in $O^+(H^{2}(K_v(A,H),\mathbb{Z}))$. The main result of this  section states that this inclusion is an equality.
\begin{prop}\label{propsection} Let $A$, $v$ and $h$ be as in Setting \ref{set}. Then
$$Mon^2(K_v(A,H))_{lt}^{pr}=O^+(H^{2}(K_v(A,H),\mathbb{Z})).$$
\end{prop}
The proof will be given in Subsection \ref{SM} using preliminary results proven in
Subsections \ref{MA} and \ref{YO}.




\subsection{Monodromy from the underlying abelian surface}\label{MA}
In this subsection we use Proposition  \ref{cor:piatesko2} and Corollary \ref{abelian_tuttoproj} to describe the monodromy operators in $Mon^{2}(K_{v}(A,H))_{lt}^{pr}$ induced by monodromy operators of $Mon^{2}(A)$.
In order to relate $Mon^{2}(K_{v}(A,H))_{lt}^{pr}$ and $Mon^{2}(A)$,
we need to recall the relation provided by the Mukai-Donaldson-Le Potier morphism (see \cite[\S 3.2]{pr_crelle}) between the cohomology of $A$ and $H^{2}(K_{v}(A,H),\ZZ)$. For every $$v=(v_0,v_2,v_4)=(2w_0,2w_2,2w_4)=2w$$ as in Setting \ref{set}, let 
$v^{\perp}=w^{\perp }\subset H^{ev}(A,\mathbb{Z}):=\oplus_{i=0}^{2}H^{2i}(A,\mathbb{Z})$ be the perpendicular lattice to $v$ in the Mukai lattice of $A$ (see \cite[Definition 1.1]{yoshi_moduli}) and let $H^{2}(K_{v}(A,H),\mathbb{Z})$  be endowed with the lattice structure given by the
Beauville-Bogomolov-Namikawa form.

By \cite[Theorem 1.7]{pr_crelle}, the  Mukai-Donaldson-Le Potier morphism 
\begin{equation} \label{mdl*} \nu_{v,H}: w^{\perp}\rightarrow H^{2}(K_{v}(A,H),\mathbb{Z})\end{equation}
is an isomorphism of abelian groups and an  isometry of lattices respecting the natural weight two  Hodge structures on $v^{\perp}$ and $H^{2}(K_{v}(A,H),\mathbb{Z})$.  

The Mukai-Donaldson-Le Potier morphism induces an identification
\begin{equation}\label{CohMuk}
O^+(w^{\perp})=O^+(H^{2}(K_{v}(A,H),\mathbb{Z})).
\end{equation}
The subgroup $SO^+_{w_2}(H^{2}(A,\mathbb{Z}))\subseteq SO^+(H^{2}(A,\mathbb{Z}))$ fixing $w_2$ is naturally a subgroup of $O^+(w^{\perp})$: the injection is given by extending every $\gamma\in SO^+_{w_2}(H^{2}(A,\mathbb{Z}))$ to the isometry of the Mukai lattice of $A$ acting as the identity on $H^0(A,\ZZ)\oplus H^4(A,\ZZ)$ and then restricting to $w^\perp$. 

In the following proposition we compare the two groups  $SO^+_{w_2}(H^{2}(A,\mathbb{Z}))$
and 
$Mon^2(K_v(A,H))_{lt}^{pr}$ as subgroups of $O^+(w^{\perp})= O^+(H^{2}(K_{v}(A,H),\mathbb{Z}))$ 
in relevant cases.
\begin{prop}\label{daabasing} Keep notation as in Setting \ref{set}. Suppose  that 
$w_2\in NS(A)$ is proportional to the class $h$ of the $v$-generic polarization $H$. Then, using the identification (\ref{CohMuk}),
$$SO^+_{h}(H^{2}(A,\mathbb{Z}))\subseteq Mon^2(K_v(A,H))_{lt}^{pr}.$$
\end{prop}
\begin{proof}
Let $f:\mathcal{A}\rightarrow T$ be an algebraic family of abelian surfaces such that $A=A_0$ is the fiber of $f$
over a point $0\in T$ and let $\mathcal{L}$ be a relatively ample line bundle on  $\mathcal{A}$ whose cohomology class restricts to a multiple of $h$ on $A$. We claim that
\begin{equation}\label{NE} Mon^{2}_f(A)\subseteq Mon^{2}(K_v(A),H)^{pr}_{lt}.\end{equation}

For $t\in T$ set $A_t:=f^{-1}(t)$ and let $L_t$ be the restriction of $\mathcal{L}$ to $A_t$.
The existence of such an $\mathcal{L}$ implies that $h$ and $v$ are fixed under the monodromy action associated with $f$ and, for $t\in T$, we denote by $v_t$ the cohomology class induced by parallel transport on $A_t:=f^{-1}(t)$. 

Let $p:\mathcal{M}_{v}\rightarrow T$ be the relative moduli space of sheaves whose fiber over $t\in T$ is the moduli space 
$M_{v_t}(A_t,L_{t})$ of the $L_t$-semistable sheaves with Mukai vector $v_t$ on $A_t$ (see \cite[Theorem 4.3.7]{HL}).  
By \cite[Corollary 4.2 and Lemma 4.4]{pr_crelle}, by removing a locally finite union of complex analytic subvarieties from $T$, we get an open subset $U$ such that  $L_t$  is a $v_t$-generic for every $t\in U$. 
By \cite[Proposition 2.16]{pr_crelle}, the morphism $p_{U}:\mathcal{M}_{v,U}\rightarrow U$ obtained from $p$ by base change over $U$ is a proper  analytically locally trivial family at every point of the domain.

We first prove the claim assuming that  $p_{U}$ admits a section. 

In this case it also admits a relative Albanese morphism and taking the inverse image of the $0$-section of the relative Albanese variety we obtain a proper analytically locally trivial family $q_{U}:\mathcal{K}_{v,U}\rightarrow U$ with projective fibers such that, for every $t\in U$, the fibre 
$q^{-1}(t)$ is the Albanese fiber $K_v(A_t,L_{t})$ of $M_{v_t}(A_t,L_{t})$.

Since $v$ and $w$ are fixed under the monodromy action associated with $f$, we can consider the local system  
$\mathcal{W}^{\perp}\subset\oplus_{i=0}^{2}R^{2i}f_{*}(\mathbb{Z})$ whose fiber at $t$ is the sublattice $w^{\perp}_t$ of the Mukai lattice of $A_t$ and compare its restriction $\mathcal{W}^{\perp}_{U}$ with the local system $R^{2}q_{U*}(\mathbb{Z})$.
For every $t\in U$ the Mukai-Donaldson-Le Potier morphism 
$\nu_{v_t,L_{t}}: w_{t}^{\perp}\rightarrow H^{2}(K_{v}(A_t,L_{t}),\mathbb{Z})$ gives an isomorphism on the corresponding fibers and moreover there exists an isomorphism of local systems 
$$\mathcal{W}_{U}^{\perp}\simeq R^{2}q_{U*}(\mathbb{Z})$$ inducing $\nu_{v_t,L_{t}}$ on the fiber over $t\in U$ (see  \cite[\S 3.2]{pr_crelle}).
Considering the monodromy groups associated with the local systems and 
using the identification \eqref{CohMuk}, it follows that $$Mon^{2}_f(A)=Mon^{2}_{q_U}(K_v(A,H))\subseteq Mon^{2}(K_v(A,H))^{pr}_{lt}.$$

In the general case where  $p_{U}$ has no section, we may replace $\mathcal{A}$ by $\mathcal{M}_{v}\times_T \mathcal{A}$
and $f$  by the induced morphism $f_{\mathcal{M}_v}: \mathcal{M}_{v}\times_T \mathcal{A}\rightarrow \mathcal{M}_{v}$. The associated relative moduli space is  given by the projection 
$$p_{\mathcal{M}_v}: \mathcal{M}_{v}\times_T \mathcal{M}_{v} \rightarrow \mathcal{M}_{v}.$$ 
Since the diagonal morphism provide a section of $p_{\mathcal{M}_v}$, the previous argument implies that  
$Mon^{2}_{f_{\mathcal{M}_v}}(A)\subseteq Mon^{2}(K_v(A,H))^{pr}_{lt}$.
Finally,  since $p:\mathcal{M}_{v}\rightarrow T$ has connected fibers, it induces a surjection on fundamental groups: it follows that  
$Mon^{2}_{f_{\mathcal{M}_v}}(A)=Mon^{2}_f(A)$ and $Mon^{2}_{f}(A)\subseteq Mon^{2}(K_v(A,H))^{pr}_{lt}$, as claimed. 

Applying the claim to the family $f_{2d}:\mathcal{A}_{2d}\rightarrow T_{2d}$, introduced in equation \eqref{intro} and using Corollary \ref{cor:piatesko2} we obtain the proposition.
\end{proof}

We are going to show that, in the important case where $w_2=0$, i.e. $v=(2,0,-2)$, the whole 
$SO^+(H^{2}(A,\mathbb{Z}))$ is contained in  $Mon^2(K_v(A,H))_{lt}^{pr}$. We argue by considering an abelian surface $A$
that is the product of two elliptic curves and analyzing how different polarizations on $A$ induce subgroups
of  $Mon^2(K_v(A,H))_{lt}^{pr}$. As a preliminary step we need to study the dependence of $K_v(A,H)$ from the $v$-generic polarization $H$.
\begin{prop}\label{prop:cham} Keep notation as in Setting \ref{set}. Assume further that 
$A$ is the product of two  elliptic curves $E_1$ and $E_2$, the Neron Severi group 
$NS(A)$ is generated by the classes $e$ and $f$ of $E_1$ and $E_2$ and $h=de+f$ for $d>5$.
If $F$ is a $H$-(semi)stable on $A$ with Mukai vector $v=(2,0,-2)$ then either 
\begin{enumerate}
\item{$F$ is $H'$-(semi)stable for every ample line bundle $H'$ on $A$ or} 
\item{$F$ fits in a non trivial extension of the form $$0\rightarrow M_1\rightarrow F\rightarrow M_2\rightarrow 0$$
where $M_1$ and $M_2$ are  line bundles on $A$ whose first Chern classes satisfy  $c_1(M_1)=e-f=-c_1(M_2)$.}
\end{enumerate}
Moreover, if $F$ is any sheaf fitting in a non trivial extension as in (2) and $ae+bf\in NS(A)$ is the class of $H'$, then  $F$ is $H'$-stable if and only if $a>b$ and is strictly $H'$-semistable if and only if $a=b$.
\end{prop}
   
\begin{proof}  
Since $H$ is $v$-generic, if $F$ is strictly $H$-semistable it fits in an exact sequence of the form 
$$0\rightarrow F_1\rightarrow F\rightarrow F_2\rightarrow 0$$ where the $F_1$ and $F_2$ are pure sheaves with Mukai vector $(1,0,-1)$.
Since $F_1$ and $F_2$ have rank one, they are stable with respect to every polarization: it follows that $F$ is strictly
$H'$-semistable    for every ample $H'$.

Let  $F$ be $H$-stable, if $F$ is also $H'$-unstable there exists an $H'$-destabilizing sequence of the form
 \begin{equation}\label{EXSE}
0\rightarrow M_1\otimes I_{Z_1}\rightarrow F\rightarrow M_2\otimes I_{Z_2} \rightarrow 0
\end{equation} 
where $M_1$ and $M_2$ are line bundles and $I_{z_1}$ and $I_{z_2}$
the  the ideal sheaves of a zero dimensional subschemes  $Z_1$ and $Z_2$. Let $l, m\in \mathbb{Z}$ be such that $c_1(M_1)=le+mf\in NS(A)$. Since  $F$ is $H$-stable and $M_1$ is a  $H'$-destabilizing subsheaf, we deduce that $lm<0$.
Using the exact sequence  (\ref{EXSE}) to relate  the Chern characters, since  $ch_2(F)=v_4=-2$ and $ch_2(M_1)=ch_2(M_2)=lm$,
we obtain $$-2=2lm-lg(Z_1)-lg(Z_2)$$       
where $lg(Z_1)$ and $lg(Z_2)$ are  the length of the subschemes $Z_1$ and $Z_2$: it follows that $lm=-1$ and $Z_1$ and $Z_2$ are empty.
Finally, since $F$ is $H$-stable and $d>1$, we conclude that $l=1$, $m=-1$ and the extension (\ref{EXSE}) is non trivial, as desired.

The final  part of the statement follows since an  exact sequence as in item (2) is $H'$-destabilizing only for $a\le b$ and 
the previous argument shows that, if $H''$ another ample line bundle, an $H''$-destabilizing sequence for $F$ would be of the same form. 
\end{proof} 
\begin{oss} According to the definition of $v$-generic polarization (see \cite[Definition 2.1]{PerRap}), an ample line bundle $H$ on $A$ 
of class $de+f$ is $(2,0,-2)$-generic if and only if $d>5$. However the statement of the proposition
holds as soon as $d>1$.
\end{oss}

Proposition \ref{prop:cham} is rephrased in term of the birational geometry of $K_v(A,H)$ in the following corollary.
\begin{cor} \label{cor:cham} Let $A$ and $v$ be as in Proposition \ref{prop:cham}, let $H_1$ and $H_2$ be $v$-generic polarizations
and let $h_1=a_1e+b_1f$ and $h_2=a_2e+b_2f$ be their respective classes.
\begin{enumerate}
\item{If $(a_1-b_1)(a_2-b_2)>0$ there exists an isomorphism $$\phi_{+}:K_v(A,H_1)\rightarrow K_v(A,H_2)$$
 sending the S-equivalence class  of a sheaf $F$ with respect to $H_1$ to the S-equivalence class of $F$ with respect to $H_2$; }
\item{if $(a_1-b_1)(a_2-b_2)<0$ there exist Zariski open subsets $U_v(A,H_1)\subset K_v(A,H_1)$ and $U_v(A,H_2)\subset K_v(A,H_2)$
obtained by removing a finite set of copies of $\mathbb{P}^3$ form the smooth loci of   $K_v(A,H_1)$ and $K_v(A,H_2)$ and an isomorphism
$$\phi_{-}:U_v(A,H_1)\rightarrow U_v(A,H_2)$$
sending the S-equivalence class  of a sheaf $F$ with respect to $H_1$  to the S-equivalence class of $F$ with respect to $H_2$.}
\end{enumerate}   
\end{cor}   
\begin{proof} If $(a_1-b_1)(a_2-b_2)>0$  the isomorphism $\phi_{+}$ exists because, by Proposition \ref{prop:cham} and irreducibility of $M_v(A,H_1)$ and $M_v(A,H_2)$ (see \cite{kls}), the subvarieties
$K_v(A,H_1)$ and $K_v(A,H_2)$ corepresent the same functor.
If $(a_1-b_1)(a_2-b_2)<0$ we define $U_v(A,H_1)\subset K_v(A,H_1)$ and $U_v(A,H_2)\subset K_v(A,H_2)$ as the open subset parametrizing sheaves that are simultaneously $H_1$-semistable and $H_2$-semistable. The same argument of the previous case implies the existence of the isomorphism $\phi_{-}$. 

It remains to show that, for $i=1,2$, the closed subset  $K_v(A,H_i) \setminus U_v(A,H_i)$ consists of a finite union of copies of $\mathbb{P}^3$ contained in the stable locus of $K_v(A,H_i)$. Exchanging $e$ and $f$,  if necessary, we may suppose $a_i>b_i$. 
By the final part of Proposition
\ref{prop:cham}, for $M_1,M_2\in Pic(A)$ such that $c_1(M_1)=e-f$ and $c_1(M_2)=f-e$, the projective space 
$\mathbb{P}(Ext^1(M_2,M_1))$ parametrizes stable sheaves and a straightforward diagram chasing shows that 
it naturally injects as a closed subvariety of the stable locus of  $M_v(A,H_i)$. Since projective spaces are contracted by the Albanese fibration, the first part of Proposition
\ref{prop:cham}  implies that $K_v(A,H_i) \setminus U_v(A,H_i)$ is a union of projective spaces of the form $\mathbb{P}(Ext^1(M_2,M_1))$.

Since $A\times A^{\vee}$ acts transitively on the fibers of the Albanese fibration of $M_v(A,H_i)$ and the action preserves stability 
with respect to every polarization, the variety  $K_v(A,H_i) \setminus U_v(A,H_i)$ is non empty. Finally 
$Ext^1(M_2,M_1)=H^1(M^{\vee}_2\otimes M_1)\simeq \mathbb{C}^4$ and the projective  variety $K_v(A,H_i) \setminus U_v(A,H_i)$ is a union of copies of $\mathbb{P}^3$ contained in the smooth locus of $K_v(A,H_i)$: since $K_v(A,H_i)$ is symplectic of dimension six this union consists of a finite set of copies of $\mathbb{P}^3$.

\end {proof}
\begin{oss} A more detailed analysis shows that  $K_v(A,H_i) \setminus U_v(A,H_i)$ consists of $256$ copies of $\mathbb{P}^3$.
\end{oss}
Corollary \ref{cor:cham} allows to improve Proposition \ref{daabasing} in the case where $v=(2,0,-2)$.
\begin{prop}\label{daabasing+}
Keep notation as in Setting \ref{set}. Suppose  that 
$v=(2,0,-2)$. Then, using the identification (\ref{CohMuk}),
$$SO^+(H^{2}(A,\mathbb{Z}))\subseteq Mon^2(K_v(A,H))_{lt}^{pr}.$$
\end{prop}
\begin{proof} We first analyse the case where $A=E_1\times E_2$, $H$ and $h=de+f$ are as in Proposition \ref{prop:cham}.
By Corollary \ref{abelian_tuttoproj} it suffices to prove that $$\bigcup_{i=d}^{d+1}SO^{+}_{ie+f}(H^2(A,\mathbb{Z}))\cup \bigcup_{i=d}^{d+1} SO^{+}_{e+if}(H^2(A,\mathbb{Z}))\subseteq Mon^2(K_v(A,H))_{lt}^{pr}.$$ 
By Proposition \ref{daabasing}, 
$$SO^{+}_{de+f}(H^2(A,\mathbb{Z})) \subseteq Mon^2(K_v(A,H))_{lt}^{pr}.$$ By (1) of Corollary \ref{cor:cham} the varieties  $K_v(A,H)$ and $K_v(A,H+E_1)$parametrize the same set of sheaves, hence they can be identified and  this identification is compatible with the Mukai-Donaldson-Le Potier morphisms.  By Proposition \ref{daabasing}
$$SO^{+}_{(d+1)e+f}(H^2(A,\mathbb{Z})) \subseteq Mon^2(K_v(A,H+E_1))_{lt}^{pr}= Mon^2(K_v(A,H))_{lt}^{pr}.$$
Let now $H'$ be an ample line bundle on $A$ whose class $h'$ equals either $e+df$ or $e+(d+1)f$. Since,  
by Proposition \ref{daabasing}, $SO^{+}_{h'}(H^2(A,\mathbb{Z})) \subseteq Mon^2(K_v(A,H'))_{lt}^{pr}$,
it remains  to show that there exists a parallel  transport operator 
$$\psi:H^2(K_v(A,H'),\mathbb{Z})\rightarrow H^2(K_v(A,H),\mathbb{Z})$$ along a proper analitically locally trivial family with projective fibers commuting with the Mukai-Donaldson-Le Potier morphisms of $K_v(A,H)$ an $K_v(A,H')$, i.e. satisfying \begin{equation}\label{MDL}\nu_{v,H}=\psi\circ\nu_{v,H'}.\end{equation} 

To define $\psi$ we use case (2) of Corollary \ref{cor:cham}: since  
$ K_v(A,H) \setminus U_v(A,H)$ and $ K_v(A,H')\setminus U_v(A,H')$ are contained in the smooth loci of 
 $ K_v(A,H)$ and $ K_v(A,H')$ and have codimension $3$, the morphism $\phi_{-}$  in  case (2) of Corollary \ref{cor:cham} induce an isomorphism 
$\phi_{-}^{*}:H^2(K_v(A,H'),\mathbb{Z})\rightarrow H^2(K_v(A,H),\mathbb{Z})$ and we set $\psi:=\phi_{-}^{*}$.

Let  $U^s_v(A,H)$ 
be the smooth locus of $U_v(A,H)$ 
and let $i:U^s_v(A,H)\rightarrow K_v(A,H)$ be the open embedding. By Lemma 3.7 of \cite{pr_crelle}
the pull-back $i^{*}:H^2(K_v(A,H),\mathbb{Z})\rightarrow H^2(U^s_v(A,H),\mathbb{Z})$ is injective
and (\ref{MDL}) follows from  \begin{equation}\label{MDL1} 
i^{*}\circ\nu_{v,H}=i^{*}\circ\psi\circ\nu_{v,H'}.
\end{equation}
By definiton of the Mukai-Donaldson-Le Potier morphism ($\S 3.2$ of \cite{pr_crelle}),  the composition $i^{*}\circ\nu_{v,H}$
is uniquely determined by any quasi-universal family on the product $A\times U^s_v(A,H)$.
Since, by case (2) of Corollary \ref{cor:cham}, the morphism $\phi_{-}$ identifies $U^s_v(A,H)$ and $U^s_v(A,H')$ and a quasi-universal family on $A\times U^s_v(A,H')$ is also quasi-universal family on $A\times U^s_v(A,H)$, equations (\ref{MDL}) and (\ref{MDL1}) hold. 

In order to show that $\psi$ is a parallel  transport operator 
 along a proper analitically locally trivial family with projective fibers, 
since  $K_{v}(A,H)$ and $K_{v}(A,H')$ are 2-factorial (\cite[Theorem 1.2]{perapimrn}), we can apply
 Theorem 1.1 of \cite{lehnpacienza} (see also \cite[Theorem 6.17]{bakkerlehn}). It implies that there exist one parameter projective analically locally  trivial deformations $\mathcal{K}\rightarrow D$ and $\mathcal{K}'\rightarrow D$ over a small disk  $D$, having $K_{v}(A,H)$ and $K_{v}(A,H')$ as fibres over $0\in D$ and  there exists a bimeromorphic morphism 
$$\phi_{D}:\mathcal{K}\dashrightarrow  \mathcal{K}'$$ over $D$,
well defined outside the indeterminacy locus of $\phi_{-}$, restricting to an isomorphism over $D\setminus\left\{0\right\}$
and to the morphism $\phi_{-}$ over $0$.
As a consequence, $\psi$ is the desired 
parallel transport operator and the proposition is proven in the case where $A$ and $H$ are as in
Proposition \ref{prop:cham}.

We now consider an abelian surface $A''$ with an ample line bundle $H''$ whose class $h''$ is primitive and such that  $h^{2}=h''^{2}>10$.
We set $v'':=(2,0,-2)\in H^{ev}(A'',\mathbb{Z})$ and assume that $H''$ is $v''$-generic.
By connectedness of the moduli space of polarized abelian surfaces there exists a family, over a connected base, of polarized abelian surfaces having $(A,h)$ and $(A'',h'')$ as fibers.
Since the construction of the Mukai-Donaldson-Le Potier morphism works in families where the 
Mukai vector $v$ stays algebraic   and the polarization stays algebraic and $v$-generic
\cite[\S 3.2]{pr_crelle}, there exists a parallel transport operator $t:H^{2}(K_v''(A'',h''),\ZZ)\rightarrow H^{2}(K_v(A,H),\ZZ)$
that extends to a monodromy operator $\tilde{t}:H^{ev}(A'',\ZZ)\rightarrow H^{ev}(A,\ZZ)$ associated with a family of polarized abelian surfaces. Since the monodromy on a family of abelian surfaces acts trivially on $H^{0}$ and 
$H^{4}$, 
and $SO^{+}(H^{2}(A'',\ZZ))$ is a normal subgroup of $O(H^{2}(A'',\ZZ))$, it follows that 
$$t^{-1}\circ SO^{+}(H^{2}(A,\ZZ))\circ t:=\left\{t^{-1}\circ \gamma\circ t\; : \; \gamma \in SO^{+}(H^{2}(A,\ZZ)) \right\}=SO^{+}(H^{2}(A'',\ZZ))$$ and $SO^{+}(H^{2}(A'',\ZZ))\subseteq Mon^2(K_v''(A'',H''))_{lt}^{pr}$.
 
In the general case, where $A$, $H$ and $h$ are as in the statement,
by density of the Noether Lefschetz locus, there exists a projective small deformation $(A'',\underline{h})$ of $(A,h)$ such that $NS(A'')$ has rank at least $2$. Letting  $v''$ and  $\underline{h}$ be the parallel transports of  $h$ and $v$,
by openness of $v$-genericity (see the appendix of \cite{pr_crelle}),  we may assume that the parallel transport  $\underline{h}$ of $h$ represents a $v''$-generic polarization $\underline{H}$ on $A''$.
Since $rk(NS(A''))\ge2$ there exists a  $v''$-generic polarization   $H''$ of  degree bigger than $10$ in the same $v''$-chamber of $\underline{H}$:
hence we have the identifications $K_{v''}(A'',\underline{H})=K_{v''}(A'',H'')$ and  $\nu_{v'',\underline{H}}=\nu_{v'',H''}$ and we obtain $$SO^{+}(H^{2}(A'',\ZZ))\subseteq Mon^2(K_{v''}(A'',H''))_{lt}^{pr}=Mon^2(K_{v''}(A'',\underline{H}))_{lt}^{pr}.$$
Finally, as there exists a family, over a connected base, of polarized abelian surfaces having $(A,h)$ and $(A'',\underline {h})$ as fibers, arguing as above we get $$SO^{+}(H^{2}(A,\ZZ))\subseteq Mon^2(K_{v}(A,H))_{lt}^{pr}.$$
\end{proof}

\subsection{Yoshioka's isomorphisms of moduli spaces}\label{YO}   
In this subsection we recall and adapt to our context, where necessary,  
two results on isomorphisms between moduli spaces 
exhibited by Yoshioka in $\S 3$ of \cite{yoshi_moduli}. 

The induced morphisms in cohomology are conveniently described by using, for every $v=2w$ as in our assumptions, the Mukai-Donaldson-Le Potier morphism 
$$\nu_{v,H}: w^{\perp}\rightarrow H^{2}(K_{v}(A,H),\mathbb{Z})$$
introduced in the previous subsection.

The isomorphism $\psi$ of the following proposition will allow to construct
monodromy operators in $Mon^2(K_v(A,H))^{pr}_{lt}\setminus O^{+}(H^{2}(A,\ZZ))$, i.e. monodromy operators that cannot be obtained  by deforming  the underlying abelian surface.
 
\begin{prop}\cite[Thm. 3.15]{yoshi_moduli}\label{prop:yoshicambia}
Let $A$, $H$ and $h$ be as in Setting \ref{set}. Assume futher that $A=E_1\times E_2$ is an abelian surface that is the product of a very general pair $(E_1,E_2)$ of elliptic curves  and let  $e$ and $f\in H^{2}(A,\ZZ)$ be the classes of the factors.  If  $h=e+kf$ for $k>>0$,  the following hold:
 
\begin{itemize}
\item[i)] There exists an isomorphism $\psi: K_{(2,0,-2)}(A,H)\rightarrow K_{(0,2e+2f,2)}(A,H).$
\item[ii)] There exists a commutative diagram \begin{equation} \label{cohom}
\xymatrix{
(1,0,-1)^{\perp} \ar[d]_{\nu_{(2,0,-2),H}} \ar[r]^{\phi} & (0,e+f,1)^{\perp} \ar[d]_{\nu_{(0,2e+2f,2),H}}\\
 H^2(K_{(2,0,-2)}(A,H)) \ar[r]^{\psi_*} & H^2(K_{(0,2e+2f,2)}(A,H))}
\end{equation}
where 
$\phi:(1,0,-1)^{\perp}\rightarrow (0,e+f,1)^{\perp}$ is the isometry
given by 
\begin{equation}\label{varphiespl}
\varphi((r,ae+bf+\alpha,s))=(-a,re-(s+a)f+\alpha,r+b)
\end{equation}
for $\alpha \in H^{2}(A,\ZZ)$  perpendicular to both $e$ and $f$.
\end{itemize}

\end{prop}
\begin{proof}
Let $\mathcal{P}$ be the normalized Poincaré line bundle on 
$E_2\times E_2$
and let $N$ be a  line bundle of degree one 
on $E_1$. Set $\mathcal{Q}:=p_{2,3}^{*}(\mathcal{P})\otimes p_1^{*}N$ 
where the morphisms\\ $p_{2,3}:E_1\times E_2\times E_2\rightarrow E_2\times E_2$ 
and $p_{1}:E_1\times E_2\times E_2\rightarrow E_1$ are 
the projections. Since  $E_1\times E_2\times E_2$ is 
naturally identified with the fibre 
product $A\times_{E_1}A\subset A\times A$, 
the sheaf $\mathcal{Q}$ can be regarded as a coherent sheaf 
on $A\times A$. 

By \cite[Thm. 3.15]{yoshi_moduli} the sheaf 
$\mathcal{Q}$ is the kernel of  a 
Fourier-Mukai transform satisfying 
the weak index theorem with 
index one on every semistable sheaf 
in $M_{(2,0,-2)}(A,H)$ and inducing  the isomorphism
$$\psi: K_{(2,0,-2)}(A,H)\rightarrow K_{(0,2e+2f,2)}(A,H).$$

In order to prove commutativity of diagram  (\ref{cohom}), we need first to compute the map   $\psi^{H^{ev}}:H^{ev}(A,\ZZ)\rightarrow H^{ev}(A,\ZZ)$  induced on the even cohomology of $A$
by the kernel $\mathcal{Q}$, i.e. 
the map given by 
$$\psi^{H^{ev}}((r, ae+bf+\alpha, s)):= 
p_{1,3*}(ch(\mathcal{Q})p_{1,2}^{*}((r, ae+bf+\alpha, s))),$$ 
where 
$p_{1,j}:E_1\times E_2\times E_2\rightarrow E_1\times E_2$ 
is the projection on the product of the first and the j-th factors.

By definition of $\mathcal{Q}$, we have 
$$\psi^{H^{ev}}((r,ae+bf+\alpha, s))=p_{1,3*}
(ch(p_{2,3}^{*}(\mathcal{P})\otimes p_1^{*}N)
p_{1,2}^{*}((r, ae+bf+\alpha, s)))=$$
$$
p_{1,3*}((p_{2,3}^{*}(ch(\mathcal{P}))
p_{1,2}^{*}((r, ae+bf+\alpha, s)(1,f,0)))=$$
$$ p_{1,3*}
((p_{2,3}^{*}(ch(\mathcal{P}))
p_{1,2}^{*}((r, ae+(b+r)f+\alpha, s+a)))$$ 
and, using the K\"unneth decomposition of the classes 
$1, e, f, \alpha$  and of the Poincar\'e dual of a point 
for the product $A=E_1\times E_2$   
and letting $\mathcal{P}$ act\footnote{Recall that $\mathcal{P}$ sends the class of a point to the fundamental class of $E_2$ and acts as $-1$ on $H^1$.} on the factors 
coming from the cohomology of $E_2$, we get 
$$ p_{1,3*}((p_{2,3}^{*}(ch(\mathcal{P}))
p_{1,2}^{*}((r, ae+(b+r)f+\alpha, s+a)))=$$$$
(a, -re+(s+a)f+-\alpha, -(b+r)).$$
It follows that 
\begin{equation}\label{-cohom}
\psi^{H^{ev}}((r, ae+bf+\alpha, s))=
(a, -re+(s+a)f+\alpha, -(b+r)).
\end{equation}
To deduce the commutativity of diagram  (\ref{cohom}) from formula (\ref{-cohom}) 
we notice that $\phi$ equals the opposite of the restriction 
$\psi^{H^{ev}}_{|(1,0,-1)^{\perp}}:(1,0,-1)^{\perp}\rightarrow (0,e+f,1)^{\perp}$
of   $\psi^{H^{ev}}$.
We recall that, 
by \cite[Lemma 3.7]{pr_crelle}, the open embedding  of the stable locus  
$i_{(2,0,-2)}:K_{(2,0,-2)}^{s}(A,H)\rightarrow  K_{(2,0,-2)}(A,H)$ and the analogous inclusion   
$i_{(0,2e+2f,2)}:K_{(0,2e+2f,2)}^s(A,H) \rightarrow  K_{(0,2e+2f,2)}(A,H)$ induce injective 
maps on  2-cohomology groups: hence it suffices to show the commutativity of the following diagram \begin{equation}
\label{meno}
\xymatrix{
(1,0,-1)^{\perp} \ar[d]_{i^{*}_{(2,0,-2)}\circ\nu_{(2,0,-2),H}} \ar[r]^{-\psi^{H^{ev}}_{|(1,0,-1)^{\perp}}} & (0,e+f,1)^{\perp} \ar[d]_{i^{*}_{(0,2e+2f,2)}\circ\nu_{(0,2e+2f,2),H}}\\
 H^2(K^s_{(2,0,-2)}(A,H)) \ar[r]^{\psi_*^s} & H^2(K^s_{(0,2e+2f,2)}(A,H))}
\end{equation}
where $\psi^s:K_{(2,0,-2)}^{s}(A,H)\rightarrow K_{(0,2e+2f,2)}^s(A,H)$ is the restriction
of $\psi$ to the stable locus.  

The commutativity of  diagram (\ref{meno}) follows by copying the proof of  
\cite[Proposition 2.4]{yoshi_moduli}.
This Proposition is stated only for primitive Mukai vectors, but its proof works in our case since, on the stable loci, the vertical arrows 
can be computed by the same formula
defining $\nu_{v',H}$ for primitive $v'$ (see $\S 3.2$ of \cite{pr_crelle}).
Finally, as in  \cite[Proposition 2.4]{yoshi_moduli},  we have the  minus in diagram  (\ref{meno}) because the sheaves parametrized by  $M_{(2,0,-2)}^{s}(A,H)$
satisfy the weak index theorem with odd index.  
\end{proof}

Using the identification (\ref{CohMuk}) 
we see 
$SO^{+}_{e+f}(H^{2}(E_1\times E_{2},\ZZ))$ as a subgroup of $O^{+}((0,e+f,1)^{\perp})$ and, keeping notation as in the previous proposition, we  set
$$\phi^{-1}\circ SO^{+}_{e+f}(H^{2}(E_1\times E_{2},\ZZ))\circ \phi:=$$$$
\left\{ \phi^{-1}\circ g\circ\phi|\; g\in SO^{+}_{e+f}(H^{2}(E_1\times E_{2},\ZZ))\subset O^{+}((1,0,-1)^{\perp}) \right\}.$$
As a consequence of Proposition \ref{prop:yoshicambia} we 
get the following corollary
\begin{cor} \label{casoExE} Under the assumption of Proposition \ref{prop:yoshicambia}, the union of subgroups 
$$SO^{+}(H^{2}(E_1\times E_{2},\ZZ))\cup \phi^{-1}\circ SO^{+}_{e+f}(H^{2}(E_1\times E_{2},\ZZ))\circ \phi$$ is contained in  
$$Mon^2(K_{2,0-2}(E_1\times E_2,H))^{pr}_{lt}\subseteq O^{+}((1,0,-1)^{\perp}).$$ 
\end{cor}
\begin{proof} The inclusion $SO^{+}(H^{2}(E_1\times E_{2},\ZZ))\subseteq Mon^2(K_{2,0-2}(E_1\times E_2,H))^{pr}_{lt}$ follows from (2) 
of Proposition \ref{daabasing}. Commutativity of diagram (\ref{cohom}) implies that
$$\phi^{-1}\circ g\circ\phi\in Mon^2(K_{(2,0-2)}(E_1\times E_2,H))^{pr}_{lt}\subseteq O^{+}((1,0,-1)^{\perp})$$
for every $g\in Mon^2(K_{(0,2e+2f,2)}(E_1\times E_2,H))^{pr}_{lt}\subseteq O^{+}((0,e+f,1)^{\perp})$.
It remains to show that  
$$SO^{+}_{e+f}(H^{2}(E_1\times E_{2},\ZZ))\subseteq Mon^2(K_{(0,2e+2f,2)}(E_1\times E_2,H))^{pr}_{lt}.$$
This does not follows directly from Proposition \ref{daabasing}  since, under the assumptions of 
Proposition \ref{prop:yoshicambia}, the intermediate component $2e+2f$ of the Mukai vector  
$(0,2e+2f,2)$ is not a multiple of the 
$(0,2e+2f,2)$-generic polarization $e+kf$ and, moreover, $e+f$ is not a 
$(0,2e+2f,2)$-generic polarization.

To deal with this problem we  argue as follows. 
By density of the Noether-Lefschetz locus there exists a small deformation $A'$
of $E_1\times E_2$ where both $e+f$ and $e+kf$ remain algebraic (i.e. $e$ and $f$ remain algebraic) along the deformation 
and $NS(A')$ has rank at least $3$. Let $e',f'\in H^{2}(A',\ZZ)$ 
and $v'=(0,2e'+2f',2)\in H^{ev}(A',\ZZ)$ be the parallel transport 
images of $e,f$ and $v=(0,2e+2f,2)$ and let $H'$ be an ample line bundle whose cohomology class is $e'+kf'$. 
By \cite[Corollary 4.2 and Lemma 4.4]{pr_crelle} $H'$ is $v'$-generic.
Since the construction of the Mukai-Donaldson-Le Potier morphism works in families where the 
Mukai vector $v$ stays algebraic   and the polarization stays algebraic and $v$-generic
\cite[\S 3.2]{pr_crelle}, there exists a parallel transport operator $$t:H^{2}(K_{(0,2e+2f,2)}(A,H),\mathbb{Z})\rightarrow H^{2}(K_{(0,2e+2f,2)}(A',H'),\mathbb{Z}),$$ coming from a family with projective fibers, that extends to a parallel transport operator $\tilde{t}:H^{ev}(A,\mathbb{Z})\rightarrow H^{ev}(A,\mathbb{Z})$ induced by a family of polarized abelian surfaces. Since 
$$t^{-1}\circ Mon^{2}(K_{(0,2e+2f,2)}(A',H'))^{pr}_{lt} \circ t\subseteq Mon^{2}(K_{(0,2e+2f,2)}(A,H))^{pr}_{lt}$$ and 
$$t^{-1}\circ  SO^{+}_{e'+f'}(H^{2}(A',\ZZ)) \circ t=SO^{+}_{e+f}(H^{2}(A,\ZZ))$$ it remains to prove that 
$$SO^{+}_{e'+f'}(H^{2}(A',\ZZ))\subseteq Mon^2(K_{(0,2e'+2f',2)}(A',H'))^{pr}_{lt}.$$

By boundedness of the Hilbert scheme the set $S$ of the classes in 
$H^{2}(A',\ZZ)$ that can be represented by subcurves of curves representing $2e'+2f'$ is finite.
Hence, in every open subset of the ample cone of $A'$, there exists a polarization $\underline{H}'$ such that the saturation $L$ of the lattice generated by $e'+f'$ and  the class $\underline{h}'$ of $\underline{H}'$ intersects $S$ only in $e'+f'$ and $2e'+2f'$.
If we take $\underline{H}'$ in the same open chamber of $H'$, we have identifications $K_{(0,2e'+2f',2)}(A',H')=K_{(0,2e'+2f',2)}(A',\underline{H}')$
and $\nu_{(0,2e'+2f',2),H'}=\nu_{(0,2e'+2f',2),\underline{H}'}$: hence it suffices to show that 
\begin{equation}\label{ff}SO^{+}_{e'+f'}(H^{2}(A',\ZZ))\subseteq Mon^2(K_{(0,2e'+2f',2)}(A',\underline{H}'))^{pr}_{lt}.\end{equation}

Let $A''$ be a small deformation of $A'$ such that $L$ remains algebraic along the deformation 
and $NS(A'')$ has rank $2$ (i.e. $L\cong NS(A'')$). 
Let $\gamma$ and $ \underline{h}''\in H^{2}(A'',\ZZ)$ be the parallel transport 
images of $e'+f'$ and $\underline{h}'$, let  $\underline{H}''$ 
be an ample line bundle whose cohomology class is $\underline{h}''$ and set 
$v''=(0,2e''+2f'',2)=(0,2\gamma,2)$.
The same argument used above implies that  
(\ref{ff}) follows from 
\begin{equation}\label{ff1}SO^{+}_{\gamma}(H^{2}(A'',\ZZ))\subseteq Mon^2(K_{(0,2\gamma,2)}(A'',\underline{H}''))^{pr}_{lt}.\end{equation}

By properness of the relative Hilbert scheme and by construction, the only cohomology classes 
of curves that are contained in curves whose class is $2\gamma$, are $\gamma$ and $2\gamma$.
By definition of $v$-genericity for Mukai vectors of dimension $1$ sheaves (see \cite[Definition 2.1]{PerRap}),
every polarization on $A''$ is $v''$-generic: hence an ample line bundle $\Gamma$ representing $\gamma$ and $\underline{H}''$ are in the same $v''$-chamber 
and,  as above, there is an identification 
$Mon^2(K_{(0,2\gamma,2)}(A'',\Gamma))^{pr}_{lt}=Mon^2(K_{(0,2\gamma,2)}(A'',\underline{H}''))^{pr}_{lt}$  
compatible with the Mukai-Donaldson-Le Potier morphisms $\nu_{(0,2\gamma,2),\Gamma}$ and $\nu_{(0,2\gamma,2),\underline{H}''}$. 
By Proposition \ref{daabasing}, $SO^{+}_{\gamma}(H^{2}(A'',\ZZ))\subseteq Mon^2(K_{(0,2\gamma,2)}(A'',\Gamma))^{pr}_{lt}$ 
and we obtain the inclusion (\ref{ff1}).
   
\end{proof}
Using the dual abelian surface and the Poincaré duality morphism, introduced in Remark \ref{dual},
the following proposition shows the existence of an isomorphism $\rho$ that will allow to construct 
elements in $Mon^2(K_v(A,H))^{pr}_{lt}$ with determinant $-1$.

\begin{prop}\cite[Prop. 3.2]{yoshi_moduli}\label{lem:yoshiduale}
Let $(A,h)$ be a  polarized abelian surface such that $NS(A)=\ZZ h$ and $h^{2}=4$ and let $H$ be an ample line bundle on $A$ whose class is $h$. Let $(\widehat{A},\widehat{h})$ be the dual abelian surface of $(A,h)$ and let $\widehat{H}$ be an ample line bundle on $\widehat{H}$
whose class is $\widehat{H}$. 
The following hold:

\begin{itemize}
\item[i)] There exists an isomorphism $\rho: K_{(2,2h,2)}(A,H)\rightarrow K_{(2,2\widehat{h},2)}(\widehat{A},\widehat{H}).$
\item[ii)] There exists a commutative diagram \begin{equation} \label{cohom2}
\xymatrix{
(1,h,1)^{\perp} \ar[d]_{\nu_{(2,2h,2),H}} \ar[r]^{\varrho} & (1,\widehat{h},1)^{\perp} \ar[d]_{\nu_{(2,2\widehat{h},2),\widehat{H}}}\\
 H^2(K_{(2,2h,2)}(A,H)) \ar[r]^{\rho_*} & H^2(K_{(2,2\widehat{h},2)}(\widehat{A},\widehat{H}))}
\end{equation}
where 
$\varrho:(1,h,1)^{\perp}\rightarrow (1,\widehat{h},1)^{\perp}$ is the isometry
given by 
\begin{equation} \label{varrho}
\varrho((r,\alpha,s))=-(s,-PD(\alpha),r).
\end{equation}

\end{itemize}
\end{prop}

\begin{proof}
Let $\mathcal{P}$ be the Poincaré line bundle on $A\times\widehat{A}$ and  
let $p:A\times\widehat{A}\rightarrow A$ and $q:A\times\widehat{A}\rightarrow \widehat{A}$
be the projections.
Let $\mathcal{G}_{\mathcal{P}}$ be the contravariant equivalence of the derived categories  $D(A)$ and $D(\widehat{A})$ of $A$ and $\widehat{A}$ defined by 
$$\mathcal{G}_{\mathcal{P}}(\cdot):={\rm R}\mathcal{H}om_{q}(p^{*}(\cdot)\otimes \mathcal{P},\mathcal{O})={\rm R}(q_{*}\circ\mathcal{H}om)(p^{*}(\cdot)\otimes \mathcal{P},\mathcal{O}).$$
We are going to show that for every sheaf  $E\in M_{(2,2h,2)}(A,H)$ the complex 
$\mathcal{G}_{\mathcal{P}}(E)$ has non zero cohomology only in degree $2$, i.e. the weak index theorem with index $2$ (WIT(2)) holds for $E$, and 
$H^{2}(\mathcal{G}_{\mathcal{P}}(E))$ is a semistable sheaf in $ M_{(2,2\widehat{h},2)}(\widehat{A},\widehat{H})$.
Since  $\mathcal{G}_{\mathcal{P}}$ is an equivalence, this implies that  sending a sheaf $E$ to  $H^2(\mathcal{G}_{\mathcal{P}}(E))$
gives an isomorphism between $M_{(2,2h,2)}(A,H)$ and $ M_{(2,2\widehat{h},2)}(\widehat{A},\widehat{H})$ and by restriction to the Albanese fibres we get the desired isomorphism 
$$\rho:K_{(2,2h,2)}(A,H)\rightarrow K_{(2,2\widehat{h},2)}(\widehat{A},\widehat{H}).$$ 
  
By \cite[Prop. 3.2]{yoshi_moduli} this result holds for $K_{(2,2h,2)}(A,H)$ replaced by the Albanese fibre $K_{(1,h,1)}(A,H)$
of the moduli space $M_{(1,h,1)}(A,H)$ of $H$-semistable sheaves on $A$ with  Mukai vector $(1,h,1)$ and   
$K_{(2,2\widehat{h},2)}(\widehat{A},\widehat{H})$ replaced by the Albanese fibre $K_{(1,\widehat{h},1)}(\widehat{A},\widehat{H})$
of the moduli space $M_{(1,\widehat{h},1)}(\widehat{A},\widehat{H})$  of $\widehat{H}$-semistable sheaves on $\widehat{A}$ with  Mukai vector $(1,\widehat{h},1)$.
 
Since the Mukai vector of  $\mathcal{G}_{\mathcal{P}}(E)$ only depends on the Mukai vector of $E$, 
it has to be $(2,2\widehat{h},2)$ for every $E\in M_{(2,2h,2)}(A,H)$ if WIT(2) holds for $E$.
By definition  $H^{i}(\mathcal{G}_{\mathcal{P}}(E))$ is 
the relative extension  $\mathcal{E}xt^{i}_{q}(p^{*}(E)\otimes  \mathcal{P},\mathcal{O})$ (see \cite{Lange})  and we need to prove that 
\begin{enumerate}
\item{$\mathcal{E}xt^{i}_{q}
(p^{*}(E)\otimes  \mathcal{P},\mathcal{O})=0$ for $i\ne 2$}
\item{$\mathcal{E}xt^{2}_{q}
(p^{*}(E)\otimes  \mathcal{P},\mathcal{O})$ is a semistable sheaf}
\end{enumerate}
for $E\in M_{(2,2h,2)}(A,H)$.

We distinguish 3 cases
\begin{enumerate}[a)]
\item{$E$ is strictly semistable}
\item{$Ext^{i}(E,L)=0$  for $i\ne2$ and $L\in\widehat{A}$}
\item{$Ext^{1}(E,L)\ne 0$ for some $L\in\widehat{A}$.}
\end{enumerate}

In case a) the sheaf $E$ fits in an exact sequence of the form
$$0\rightarrow E_1\rightarrow E \rightarrow E_2\rightarrow 0$$
with $E_1$ and $E_2$ in $M_{(1,h,1)}(A,H)$: 
1) and 2) follows from  \cite[Prop. 3.2]{yoshi_moduli}.

In case b)  
we study the sheaves 
$\mathcal{E}xt^{i}_{q}(p^{*}(E)\otimes  \mathcal{P},\mathcal{O})$ 
by using the  base change theorem for 
relative Ext-sheaves \cite[Thm 1.4]{Lange}.
By stability  $Hom(E\otimes M,\mathcal{O})=0$ for $M\in \widehat{A}$
and, since  $Ext^{1}(E\otimes M,\mathcal{O})=Ext^{1}(E,M^{\vee})=0$,
we deduce that the dimension of $Ext^{2}(E\otimes M,\mathcal{O})$
is $ext^{2}(E\otimes M,\mathcal{O})=\chi(M)=2$.
Theorem 1.4. of \cite{Lange} implies (1) and that $\mathcal{E}xt^{2}_{q}
(p^{*}(E)\otimes  \mathcal{P},\mathcal{O})$ 
is a rank two vector bundle with Mukai vector $(2,2\widehat{h},2)$.
 
In this case the sheaf $\mathcal{E}xt^{2}_{q}
(p^{*}(E)\otimes  \mathcal{P},\mathcal{O})$ 
is actually stable. A destabilizing quotient 
should be properly contained in a line bundle  
whose first Chern class is $\widehat{h}$: hence, if 
 $\mathcal{E}xt^{2}_{q}
(p^{*}(E)\otimes  \mathcal{P},\mathcal{O})$ is unstable, 
it admits a non zero morphism to a line bundle
$\widehat{G}$ with Mukai vector $(1,\widehat{h},1)$.
By \cite[Prop. 3.2]{yoshi_moduli}, by applying the inverse of $\mathcal{G}_{\mathcal{P}}$, there would 
exist a line bundle $G$ on $A$, with Mukai vector 
$(1,h,1)$, such that 
$\mathcal{G}_{\mathcal{P}}(G)=\widehat{G}$ and, 
since $\mathcal{G}_{\mathcal{P}}$ is an equivalence, 
$G$ would have  a non trivial morphism to $E$ making $E$ unstable.

In case c) let us first assume that $E$ is a vector bundle.
Under this  assumption there exists $L\in\widehat{A}$ and a non trivial extension
\begin{equation}\label{*}0\rightarrow L \rightarrow F\rightarrow E\rightarrow 0\end{equation} 
such that $F$ is stable. 

In fact, as $F$ has rank $3$, its stability may be checked by considering only rank 1 subsheaves and rank 1 quotients.  
The saturation of every rank one subsheaf of $F$ 
is either $L$ or a line bundle that injects into $E$: its first Chern class cannot be strictly positive by stability of $E$.  Every rank 1 quotient of $F$ has to have strictly positive first Chern class by stability of $E$ and because the extension is not trivial. It follows that $F$ 
is a stable sheaf 
with Mukai vector $(3,2h,2)$

By \cite[Thm 3.7]{yoshiFM2}  the sheaf $F$ 
satisfies  WIT(2) with respect to 
$\mathcal{G}_{\mathcal{P}}$ and    
$\widehat{F}:=\mathcal{G}_{\mathcal{P}}(F)$
is a stable sheaf on $\widehat{A}$ with Mukai vector 
$(2,2\widehat{h},3)$.
By applying $\mathcal{G}_{\mathcal{P}}$ to the exact sequence (\ref{*}), since WIT(2) holds for $L$ too, we get 
(1) and the exact sequence 
$$0\rightarrow  \mathcal{E}xt^{2}_{q}
(p^{*}(E)\otimes  \mathcal{P},\mathcal{O})\rightarrow \widehat{F}\rightarrow \mathbb{C}_{L}\rightarrow 0:$$ in particular $\mathcal{E}xt^{2}_{q}
(p^{*}(E)\otimes  \mathcal{P},\mathcal{O})$ is torsion free.
This sheaf is also stable: otherwise it would  admit a non trivial map to a stable sheaf on $\widehat{A}$ with Mukai vector $(1,\widehat{h},1)$ 
and, since $\mathcal{G}_{\mathcal{P}}$ is an equivalence inducing 
an isomorphism between $M_{(1,h,1)}(A,H)$ and $M_{(1,\widehat{h},1)}(A,H)$  (see \cite[Prop. 3.2]{yoshi_moduli}),
this would imply the existence of a non trivial map 
from a  stable sheaf with Mukai vector 
$(1,h,1)$  to $E$, contradicting stability of $E$.

In the remaining case, where $E$ is stable and  non locally free, torsion freeness of $\mathcal{E}xt^{2}_{q}
(p^{*}(E)\otimes  \mathcal{P},\mathcal{O})$ requires a different argument. The  double dual $E^{\vee\vee}$ of the stable sheaf $E$ is a $\mu$-stable vector bundle with Mukai vector $(2,2h,2+i)$ for $i>0$. 
Starting from observing that the dimension of the moduli space of  $H$-semistable sheaves on $A$ with Mukai vector $(2,2h,2+i)$ 
is 
$$\rm{dim}(M_{(2,2h,2+i)}(A,H))=(2,2h,2+i)^2=16-4(2+i)+2=10-4i$$ and can be  non negative only for 
$i\le 2$, one can check that  $E^{\vee\vee}$ fits in an exact sequence of the form
$$ 0\rightarrow L_{1}\otimes 
H\rightarrow 
E^{\vee\vee}\rightarrow L_{2}\otimes H\rightarrow 0$$ 
where $L_{i}\in \widehat{A}$.

As a consequence, using stability of $E$,
the cokernel $E/E^{\vee\vee}$ is the structure sheaf of a length two   
subscheme $Z\in A$ and we have a non trivial extension
$$0\rightarrow  L_{1}\otimes 
H\otimes I_{Z}\rightarrow E\rightarrow L_{2}\otimes H\rightarrow 0,$$ 
where $I_{Z}$ is the sheaf of ideals of $Z$ in $A$.

By \cite[Proposition  3.11]{yoshiFM2} and \cite[Thm 3.7]{yoshiFM2} 
$L_{1}\otimes 
H\otimes I_{Z}$ and $L_2$ 
satisfy the weak index theorem with index two: 
hence the same holds for $E$ and there is  a non trivial extension
$$0\rightarrow B_1 \rightarrow \mathcal{E}xt^{2}_{q}
(p^{*}(E)\otimes  \mathcal{P},\mathcal{O})\rightarrow B_2 \rightarrow 0$$
where $B_1$ and $B_2$ are $H$-stable sheaves on $A$ and their Mukai vectors
are $(2,\widehat{h},1)$ and $(0,\widehat{h},1)$, respectively. 
Since $\mathcal{H}om(B_2,B_1)=0$, the extension induces a non trivial section
of the pure one dimensional sheaf $\mathcal{E}xt^{1}(B_2,B_1)$: as a consequence 
$\mathcal{E}xt^{2}_{q}
(p^{*}(E)\otimes  \mathcal{P},\mathcal{O})$ is a vector bundle 
outside a zero dimensional  subset.

The sheaf $\mathcal{E}xt^{2}_{q}
(p^{*}(E)\otimes  \mathcal{P},\mathcal{O})$ is actually torsion free:
otherwise it would contain the structure 
sheaf of a point as a subsheaf and this is impossible since $\mathcal{G}_{\mathcal{P}}$ induces an isomorphism 
$Hom(\mathbb{C}_p, \mathcal{E}xt^{2}_{q}
(p^{*}(E)\otimes  \mathcal{P},\mathcal{O}))\simeq Hom(E,L_{p})$ and $Hom(E,L_{p})=0$ by stability.

Stability of $\mathcal{E}xt^{2}_{q}
(p^{*}(E)\otimes  \mathcal{P},\mathcal{O})$ follows 
as in the locally free case.

The commutativity of the diagram (\ref{cohom2}) is formally identical to the commutativity of the diagram (\ref{cohom})
in  Proposition \ref{prop:yoshicambia}. In this case the morphism 
$$\rho^{H^{ev}}:H^{ev}(A,\ZZ)\rightarrow H^{ev}(\widehat{A},\ZZ),$$ induced in Cohomology by $\mathcal{G}_{\mathcal{P}}$, satisfies $\rho^{H^{ev}}(r,\alpha,s)=(s,-PD(\alpha),r)$ (see \cite[Lemma 3.1]{yoshi_moduli}) and, using 
\cite[Proposition 2.5]{yoshi_moduli}, we get that the restriction $$\varrho:(1,h,1)^{\perp}\rightarrow (1,\widehat{h},1)^{\perp}$$ of the opposite 
of $\rho^{H^{ev}}$ to $(1,h,1)^{\perp}$ 
makes the diagram (\ref{cohom2}) commutative.
\end{proof}

As a consequence, we get the the existence of a monodromy operator whose determinant is $-1$.
\begin{cor}\label{det-1}.
Under the assumptions of Proposition \ref{lem:yoshiduale},
there exists $$m\in Mon^2(K_{(2,2h,2)}(A,H))^{pr}_{lt}\setminus SO^{+}((1,h,1)^{\perp}).$$
\end{cor}
\begin{proof}
In order to discuss the monodromy operator $m$ we make use of the inclusion 
$H^{2}(K_{(2,2h,2)}(A,H),\mathbb{Z})=(2,2h,2)^{\perp}\subset H^{ev}(A,\mathbb{Z})$ and the analogous  inclusion
 $H^{2}(K_{(2,2\widehat{h},2)}(\widehat{A},\widehat{H}),\mathbb{Z})=(2,2\widehat{h},2)^{\perp}\subset H^{ev}(\widehat{A},\mathbb{Z})$
induced by the Mukai-Donaldson-Le Potier morphisms.

By connectedness of moduli spaces of polarized abelian surfaces of degree $4$, and using  
the Mukai-Donaldson-Le Potier morphism for families, there exists a parallel transport operator 
$$t:  H^{2}(K_{(2,2\widehat{h},2)}(\widehat{A},\widehat{H}),\mathbb{Z})\rightarrow H^{2}(K_{(2,2h,2)}(A,H),\mathbb{Z}) ,$$
coming from a projective family, that admits an extension  to a parallel transport operator  
$\tilde{t}:H^{ev}(\widehat{A},\ZZ)\rightarrow H^{ev}(A,\ZZ)$ induced by a family of polarized abelian surfaces.

Set $m:=t\circ\rho_*$ and let 
$\widetilde{\varrho}:H^{ev}(A,\ZZ)\rightarrow H^{ev}(\widehat{A},\ZZ)$ be the extension of $\varrho$ defined by the same formula 
(\ref{varrho}), so that $\widetilde{m}:=\tilde{t}\circ \widetilde{\varrho}: H^{ev}(A,\ZZ)\rightarrow H^{ev}(A,\ZZ)$ extends $m$.
By Lemma 3 of \cite{shioda} the morphism $PD:H^{2}(A,\ZZ)\rightarrow H^{2}(\widehat{A},\ZZ)$ is not compatible with the canonical orientations on 
$H^{2}(A,\ZZ)$ and  $H^{2}(\widehat{A},\ZZ)$. On the other hand  
since the restriction $\tilde{t}_2: H^{2}(\widehat{A},\ZZ)\rightarrow  H^{2}(A,\ZZ)$ of $\tilde{t}$ is a parallel transport operator induced by family of abelian surfaces it sends the orientation on $H^{2}(A,\ZZ)$ to the orientation on $H^{2}(\widehat{A},\ZZ)$.
It follows that the determinant of  $\tilde{t}_{2}\circ PD$ is $-1$. By Formula (\ref{varrho}) the determinant of $\widetilde{m}$ is $1$ and since 
$\widetilde{m}((1,h,1))=- (1,h,1)$, its restriction $m$ to $(1,h,1)^{\perp}$ has determinant $-1$.
 \end{proof}

\subsection{Proof of Proposition \ref{propsection}}\label{SM}

The proof of this result contains a computational part. In the following remark we collect elementary facts on the
isometries that we will use.
\begin{oss}\label{oss:contivettore}
For every $w\in H^{ev}(A,\ZZ)$ such that $w^2=2$ and every $\gamma \in O(w^{\perp})$ there exists a unique 
$\widetilde{\gamma}\in O(H^{ev}(A,\ZZ))$ extending $\gamma$ and such that $\widetilde{\gamma}(w)=w$ (see Remark \ref{est2}).
If $w=(1,0,-1)$ we necessarily have $$\widetilde{\gamma}(1,0,1)=(2m+1,2\alpha,2m+1)$$ with 
$\alpha^2=2m(m+1)$ and $\widetilde{\gamma}(1,0,0)=(m+1,\alpha,m)$, $\widetilde{\gamma}(0,0,1)=(m,\alpha,m+1)$. 

Indeed, we can write $\widetilde{\gamma}(1,0,1)=(l,\beta,l)$ as it must be orthogonal to the element 
$\widetilde{\gamma}(1,0,-1)=(1,0,-1)$. Now by linearity we have:
$$ \widetilde{\gamma}(2,0,0)=(l+1,\beta,l-1). $$ Thus, $\beta=2\alpha$ and $l-1=2m$, therefore we also have
$$\widetilde{\gamma}(0,0,2)=(2m,2\alpha,2m+2).$$
And our claim follows by computing the square of these elements. Notice moreover that $\alpha$ and $2m+1$ are coprime as $(1,0,1)$ is indivisible.
\end{oss} 
We first show that $Mon^2(K_v(A))_{lt}^{pr}$ contains all isometries preserving 
the orientation of the positive cone of $H^{2}(K_v(A),\mathbb{Z})$ ( see Remark \ref{positive cone}) and having determinant $1$.
\begin{prop}\label{prop:so_sing} Let $A$, $v$, $H$ and $h$ be as in Setting \ref{set}. Then, 
using the identification \eqref{CohMuk},
$$SO^{+}(w^{\perp})=SO^{+}(H^{2}(K_v(A,H),\ZZ))\subseteq Mon^2(K_v(A,H))_{lt}^{pr}.$$
\end{prop}
\begin{proof} By \cite[Theorem 1.6, Proposition 2.16]{pr_crelle} or \cite[Theorem 1.17, Remark 1.18]{PerRap}
every two varieties of the form $K_v(A,H)$ appear as fibres of a projective family, over a connected base, 
which is an analytically 
locally trivial deformation at every point of the domain.

Hence it suffices to prove the statement in the special case where
$A=E_1\times E_2$, the Neron Severi group $NS(A)$ is generated by the classes $e$ and $f$ of the curves $E_1$ and $E_2$, 
the Mukai vector $v$ is $(2,0,-2)$ and the polarization is $e+kf$ for a big $k$ (as in Proposition \ref{prop:yoshicambia}).

Let $G\subseteq O^{+}((1,0,-1)^{\perp})$ be the subgroup generated by
$SO^{+}(H^{2}(E_1\times E_{2},\ZZ))$ and $\phi^{-1}\circ SO^{+}_{e+f}(H^{2}(E_1\times E_{2},\ZZ))\circ \phi$.
By Corollary \ref{casoExE}, the statement follows from the inclusion
$$SO^{+}(w^{\perp})\subseteq G.$$ 

We are going to show that for every 
$\gamma\in SO^{+}((1,0,-1)^{\perp})$ there exists $\delta$ belonging to $G$ such that
$\gamma((1,0,1))=\delta((1,0,1))$. 

If this is the case, letting $\widetilde{\gamma}$ and $\widetilde{\delta}$ be the extensions of 
$\gamma$ and $\delta$ as in Remark \ref{oss:contivettore}, the composition 
$\widetilde{\delta}^{-1}\circ\widetilde{\gamma}\in O(H^{ev}(E_1\times E_{2},\ZZ))$ has determinant $1$ and is the identity on 
$H^{0}(E_1\times E_{2},\ZZ)\oplus H^{4}(E_1\times E_{2},\ZZ)$: this implies that $\widetilde{\delta}^{-1}\circ\widetilde{\gamma}$ acts on $H^{2}(E_1\times E_{2},\ZZ)$ with determinant $1$. Moreover $\delta^{-1}\circ\gamma$ preserves the orientation of the positive cone of the lattice $(1,0,-1)^{\perp}$, hence the same holds for its restriction to  $H^{2}(E_1\times E_{2},\ZZ)$.  It follows that $\delta^{-1}\circ\gamma\in SO^{+}(H^{2}(E_1\times E_{2},\ZZ))$
and $\gamma\in G$ since $\delta\in G$.
 
Given $\gamma\in SO^{+}((1,0,-1)^{\perp})$, in order to find $\delta$ we notice, from the definition of $\varphi$ in equation (\ref{varphiespl}), that 
\begin{equation}\label{x}\varphi(l,\chi+bf,l)= (0,l(e-f)+\chi,l+b)\end{equation} for every $l,b\in\ZZ$ and $\chi\in H^{2}(E_1\times E_{2},\ZZ)$ perpendicular to both $e$ and $f$. 

By Remark \ref{oss:contivettore}, we know $\gamma((1,0,1))=(2m+1,2\alpha,2m+1)$. 

Let us first suppose that $\alpha$ is primitive.
By equation (\ref{x}) we have $$\varphi((1,0,1))=(0,e-f,1)$$ and, by Eichler's criterion \ref{lem:eichler} and Remark \ref{est2},
there exists an isometry 
$g$ belonging to $SO^{+}(H^{2}(E_1\times E_{2},\ZZ))$ such that $g(e+f)=e+f$ and $g(e-f)=(2m+1)(f-e)+2\rho$, with $\rho$ primitive and orthogonal to $e$ and $f$.

We have $(g\circ\varphi) ((1,0,1))=g(0,f-e,1)=(0,(2m+1)(e-f)+2\rho,1)$ and, by equation (\ref{x}),
$$\varphi^{-1} \circ g \circ \varphi((1,0,1))= (2m+1,2(mf+\rho),2m+1). $$
Notice that $\rho':=\rho+mf$  is primitive as $\rho$ is and $f$ is orthogonal to it. We again apply Eichler's criterion \ref{lem:eichler} to find an isometry $g'\in SO^{+}(H^{2}(E_1\times E_{2},\ZZ))$ such that $g'(\rho')=\alpha$.
It follows that  $\delta:=g'\circ \varphi^{-1} \circ g \circ \varphi\in G$ satisfies $\delta((1,0,1))=\gamma((1,0,1))$.

Let us now suppose that $\alpha=n\beta$, with $n \neq \pm 1$. By Eichler's criterion there exists an element 
$g''\in SO^{+}(H^{2}(E_1\times E_{2},\ZZ))\subseteq G$ such that $g''(\beta)=\beta '\in\langle e,f\rangle^\perp$. 
Therefore $\varphi (2m+1,2n\beta ',2m+1)=
(0,(2m+1)(e-f)+2n\beta ',2m+1)$. As $\beta '$ is orthogonal to $f-e$, by applying Eichler's criterion again, there exists an element 
$g'''\in SO^{+}_{e+f}(H^{2}(E_1\times E_{2},\ZZ))$ fixing the element $e+f$ and  such that 
$g'''((2m+1)(e-f)+2n\beta')=(2m+1)(e-f)+2\beta''$, with $\beta''$ primitive and perpendicular to $e$ and $f$. Let now 
$h:=\varphi^{-1}\circ g'''\circ \varphi \circ g''\in G$. We have 
$$(h\circ\gamma)(1,0,1) =h ((2m+1,2n\beta',2m+1))= (2m+1,2\beta'',2m+1) $$
and, by  the first part of the proof, the equality  $\delta((1,0,1))=(2m+1,2\beta',2m+1)$ holds for some  $\delta\in G$.
We conclude that $\gamma ((1,0,-1))=(h^{-1}\circ \delta) ((1,0,-1))$
and since $h^{-1}\circ \delta \in G$, this finishes the proof.
\end{proof}

\begin{oss}\label{quad2}
The argument of the above proof could be applied to construct monodromy operators also for generalized Kummer manifolds $K_{n}(A)$ of dimension $2n$ for $n>1$. In this case $$SO^{+}((1,0,-1-n)^\perp)\not \subseteq Mon^{2}(K_{n}(A)),$$ and the above arguments only show that 
$Mon^{2}(K_{n}(A))$ contains 
the  subgroup   
$S\widetilde{O}^+((1,0,-1-n)^\perp)$ consisting of isometries that  can be extended to the whole Mukai lattice with trivial action $(1,0,-1-n)$.
This difference has a pure lattice theoretic explanation: every isometry of $(1,0,-1-n)^\perp$ extends to an isometry of the Mukai lattice fixing $(1,0,-1-n)$ if and only if $n=0$.
Using the argument of the proof of the previous Proposition and the analogue of  
Corollary \ref{det-1} based on \cite[Proposition 3.2]{yoshi_moduli}
one can prove that  $Mon^{2}(K_{n}(A))$ contains an extension of 
$S\widetilde{O}^+((1,0,-1-n)^\perp)$ of index two.
However,  to prove that this extension is  the monodromy group of generalized Kummer manifold,  one would still need some argument to show that a monodromy operator  extends to an isometry of the Mukai lattice (see \cite{mark_mon,mon_mon,kier_mon}).  
\end{oss}

Combining Proposition \ref{prop:so_sing} and Corollary \ref{det-1} we can finally prove that 
for every $A,v=2w,H$ and $K_v(A)$ as in Setting \ref{set},
\begin{equation}\label{MSING}Mon^2(K_v(A))_{lt}^{pr}=O^{+}(H^{2}(K_v(A),\ZZ))=O^{+}(w^{\perp}).\end{equation}
Since $SO^{+}(w^{\perp})$ has index two in $O^{+}(w^{\perp})$
and  
$SO^{+}(w^{\perp})\subseteq Mon^2(K_v(A))_{lt}^{pr}$ by Proposition \ref{prop:so_sing},
it suffices to find a monodromy operator in $Mon^2(K_v(A))_{lt}^{pr}$ with determinant $-1$. 
As in the proof of Proposition \ref{prop:so_sing}, by \cite[Theorem 1.6, Proposition 2.16]{pr_crelle} or \cite[Theorem 1.17,  Remark 1.18]{PerRap}, it is enough to find this element in a particular case.
This has been done in Corollary \ref{det-1}.
\qed

\section{Total Monodromy and the Classical Bimeromorphic Global Torelli}\label{sec:monotutta}
In this section we prove that the monodromy group of a \hk manifold of $OG6$ type is maximal
and that Classical Bimeromorphic Global Torelli Theorem holds for this class of manifolds. 

In order to show these results we will only use monodromy operators coming from the singular models 
studied in the previous subsection and one more monodromy operator induced by a specific prime exceptional divisor or a divisorial contraction. 

Let $A,v=2w,h$ and $K_{v}(A,H)$ be as in Setting \ref{set} and let 
$\pi_{v}:\widetilde{K}_{v}(A,H)\rightarrow K_{v}(A,H)$
be the blow up of $K_v(A,H)$ along its singular locus with reduced structure.
By \cite[Theorem 1.6]{pr_crelle}, the variety  $\widetilde{K}_{v}(A,H)$ is a \hk manifold of $OG6$ type.
Moreover, by \cite[Theorem 2.4, Remark 3.3]{perapimrn},  
$$\pi_{v}^{*}:H^{2}(K_{v}(A,H),\ZZ)\rightarrow H^{2}(\widetilde{K}_{v}(A,H),\ZZ)$$
is a Hodge isometric embedding and 
$$H^{2}(\widetilde{K}_{v}(A,H),\ZZ)=\pi_{v}^{*}(H^{2}(K_{v}(A,H),\ZZ))\oplus_{\perp}\ZZ \epsilon_v$$
where $\epsilon_v$ is a  class such that $\epsilon_v^2=-2$ and $2\epsilon_v$ equals the class  of the (irreducible) exceptional divisor.
Using the Mukai-Donaldson-Le Potier  isomorphism $w^{\perp}\simeq H^{2}(K_{v}(A,H),\ZZ)$,
we have  a Hodge isometric  isomorphism (\cite[Theroem 3.1]{perapimrn})
\begin{equation} \label{Hiso}
H^{2}(\widetilde{K}_{v}(A,H),\ZZ)\simeq w^{\perp}\oplus_{\perp}\ZZ \epsilon_v
\end{equation}
that allows the identification 
\begin{equation}\label{CohMuk2} O^+(H^{2}(\widetilde{K}_{v}(A,H),\ZZ))=O^{+}(w^{\perp}\oplus_{\perp}\ZZ \epsilon_v).\end{equation}
Hence 
we  regard the group $Mon^{2}(\widetilde{K}_{v}(A,H))$ 
as a subgroup of $O^{+}(w^{\perp}\oplus_{\perp}\ZZ \epsilon_v)$.

The following proposition describes the contribution of $Mon^{2}(\widetilde{K}_{v}(A,H))^{pr}_{lt}$ to $Mon^{2}(\widetilde{K}_{v}(A,H))$.
\begin{prop}\label{dasingahk} Using the identification \eqref{CohMuk2},
$$O^{+}(w^{\perp})\subseteq Mon^{2}(\widetilde{K}_{v}(A,H))\subseteq O^{+}(w^{\perp}\oplus_{\perp}\ZZ \epsilon_v).$$
\end{prop}
\begin{proof} Proposition \ref{propsection} gives the identification $O^{+}(w^{\perp})= Mon^{2}(K_{v}(A,H))^{pr}_{lt}$.
By functoriality of the blow up and injectivity of $\pi_{v}^{*}$ we have 
 $Mon^{2}(K_{v}(A,H))^{pr}_{lt}\subseteq Mon^{2}(\widetilde{K}_{v}(A,H))$ and, since 
 the exceptional divisor of $\pi_v$ is irreducible, the monodromy group $Mon^{2}(K_{v}(A,H))^{pr}_{lt}$ acts trivially on $\epsilon_v$
as desired.
\end{proof}

Monodromy operators in $Mon^{2}(\widetilde{K}_{v}(A,H))$ of a different nature can be obtained by considering prime exceptional divisors on particular projective \hk manifolds of $OG6$ type.
\begin{defn}\label{ndef}
A prime exceptional divisor on a \hk manifold $X$ is a reduced and irreducible divisor of negative Beauville-Bogomolov square.
\end{defn}
By \cite[Prop. 6.2]{mark_tor} for every  prime exceptional divisor $D$ on $M$
and every class $\alpha\in H^{2}(M,\ZZ)$, the value $2\frac{([D],\alpha)}{([D]^{2})}$ is integral
and the associated reflection $$R_D:H^{2}(M,\ZZ)\rightarrow H^{2}(M,\ZZ)$$ defined by the formula
$$R_D(\alpha)=\alpha-2\frac{([D],\alpha)}{[D]^2}[D] $$
is a monodromy operator in $Mon^{2}(M)$.

The projective model of $OG6$ that we need is the original O'Grady example and 
the modular image of the  corresponding prime exceptional divisor parametrizes non locally free sheaves.

\begin{prop}\label{prop:nagai}
Let $A$ be the Jacobian of a generic curve of genus $2$ and let $v=(2,0,-2)\in H^{ev}(A,\ZZ)$. 
Let $\widetilde{B}\subset \widetilde{K}_{(2,0,-2)}(A,H)$ be the strict transform in $\widetilde{K}_{(2,0,-2)}(A,H)$
of the locus of non locally free sheaves in $K_{(2,0,-2)}(A,H)$.
\begin{enumerate}
\item{The subvariety $\widetilde{B}\subset \widetilde{K}_{(2,0,-2)}(A,H)$ is a prime exceptional divisor and the Beauville-Bogomolov square of the associated class is $-4$.}
\item{The reflection $R_{\widetilde{B}}:H^{2}(\widetilde{K}_{(2,0,-2)}(A,H),\ZZ)\rightarrow \widetilde{K}_{(2,0,-2)}(A,H),\ZZ)$ belongs to the monodromy group  $Mon^{2}(\widetilde{K}_{(2,0,-2)}(A,H))$.}
\end{enumerate}
\end{prop}
\begin{proof}
Irreducibility of $\widetilde{B}$ follows from \cite[Lemma 4.3.3]{OG6} and $[\widetilde{B}]^{2}=-4$ is proven in \cite[Theorem 3.5.1]{rap_phdz} or in \cite[Theorem 9.1]{per}: this proves (1).
(2) follows from (1) and  \cite[Prop. 6.2]{mark_tor}.
\end{proof}

\begin{oss} Nagai \cite[Main Theorem]{nagai} proved furthermore that there exists a divisorial contraction $\widetilde{K}_{(2,0,-2)}(A,H)$ whose contracted locus is $\widetilde{B}$. 
\end{oss}
 Our main result follows from Proposition \ref{dasingahk} and Proposition \ref{prop:nagai}
\begin{thm}\label{thm:monodromy}
\begin{enumerate}
\item{The monodromy group of a \hk manifold $X$ of $OG6$ type is maximal, i.e.
$$Mon^{2}(Y)=O^{+}(H^{2}(Y,\ZZ)).$$}
\item{The Classical Bimeromorphic Global Torelli Theorem holds for \hk manifolds of $OG6$ type, i.e.
two \hk manifolds $X'$ and $X''$ of  $OG6$ type are bimeromorphic if and only if there exists an isometric  isomorphism of Hodge structures
between $H^{2}(X',\ZZ)$ and $H^{2}(X'',\ZZ)$.}
\end{enumerate}
\end{thm}
\begin{proof}
(1) By definition of deformation equivalence type of \hk manifolds it suffices to prove the statement 
for the original O'Grady example $\widetilde{K}_{(2,0,-2)}(A,H)$ where $A$ is the Jacobian of a generic curve of genus $2$.
In this case we have 
\begin{eqnarray}\label{eq:cohom}
H^{2}(\widetilde{K}_{(2,0,-2)}(A,H))=(1,0,-1)^{\perp}\oplus_{\perp} \ZZ \epsilon=\\\nonumber H^2(A,\ZZ)\oplus_{\perp}\ZZ\zeta\oplus_{\perp}\ZZ \epsilon,
\end{eqnarray}
where $\epsilon:=\epsilon_{(2,0,-2)}$ and $\zeta^{2}=\epsilon^2=-2.$
By \cite[Theorem 3.5.1]{rap_phdz} the class $[\widetilde{B}]$ of the prime exceptional divisor is in $\ZZ\zeta\oplus_{\perp}\ZZ \epsilon$
and since $[\widetilde{B}]^2=-4$, up to changing the signs of $\zeta$ and $\epsilon_{(2,0,-2)}$, we may suppose that $[\widetilde{B}]=\zeta+\epsilon$.

By Proposition \ref{dasingahk} we know that $O^{+}((1,0,-1)^{\perp})\subseteq Mon^{2}(\widetilde{K}_{(2,0,-2)}(A,H))$ and by Proposition \ref{prop:nagai} we know that $R_{\widetilde{B}}:=R_{\zeta+\epsilon}\in Mon^{2}(\widetilde{K}_{(2,0,-2)}(A,H))$: hence the statement follows if we prove that the the subgroup generated by $R_{\zeta+\epsilon }$ and $$O^{+}((1,0,-1)^{\perp})=O^{+}(H^2(A,\ZZ)\oplus_{\perp}\ZZ\zeta)$$ is the whole $$O^{+}(H^2(A,\ZZ)\oplus_{\perp}\ZZ\zeta\oplus_{\perp}\ZZ \epsilon).$$

We have to show that every $f\in O^{+}(H^2(A,\ZZ)\oplus_{\perp}\ZZ\zeta\oplus_{\perp}\ZZ \epsilon)$
can be obtained by composing    $R_{\zeta+\epsilon}$ and elements of $O^{+}((1,0,-1)^{\perp})=O^{+}(H^2(A,\ZZ)\oplus_{\perp}\ZZ\zeta)$.

First, notice that $R_{\zeta+\epsilon }$ sends $\zeta$ in $-\epsilon$, $\epsilon$ to $-\zeta$ and leaves their orthogonal invariant.
Since the divisibility of  $\epsilon$ is $2$, its image has to be of the form $$f(\epsilon)=2u+a\zeta+b\epsilon$$ 
for $a,b\in\ZZ$ and $u\in H^2(A,\ZZ)\simeq U^3$. 

We split the proof in four steps:
\begin{itemize}
\item[$b=0$] This means that $2u+a\zeta$ is primitive and has divisibility $2$ inside the lattice $H^2(A,\ZZ)\oplus_{\perp} \ZZ\zeta$. By Eichler's criterion (cf. Lem. \ref{lem:eichler}), 
there is an isometry $g\in O^{+}(H^2(A,\ZZ)\oplus_{\perp} \ZZ\zeta)=O^+((1,0,-1)^{\perp})$ such that 
$g(2u+a\zeta)=-\zeta$, therefore the isometry $h:=R_{\zeta+\epsilon}\circ g\circ f$ sends $\epsilon$ into itself, 
hence $h\in O^{+}(H^2(A,\ZZ)\oplus_{\perp} \ZZ\zeta)$. Since $g^{-1}\circ R_{\zeta+\epsilon}\circ h=f$, our claim holds. 
\item[$a=0$] This case reduces to the first one after composing $f$ with $R_{\zeta+\epsilon}$. 
\item[$a,b\neq 0$ and $2\not| \;u $] Since $(f(\epsilon))^2=\epsilon^2=-2$, either $a$ or $b$ is even and, since we can compose 
with $R_{\zeta+\epsilon}$, we can suppose that $a=2c$. As $u$ is not divisible by $2$, a primitive sub multiple of the element $u+c\zeta$ has divisibility $1$. Therefore, by Eichler's criterion (Cf. Lem. \ref{lem:eichler}), there exists 
$g\in O^{+}(H^2(A,\ZZ)\oplus_{\perp} \ZZ\zeta)$ such that $g(u+c\zeta)=\widetilde{u}$ for an element $\widetilde{u}\in H^2(A,\ZZ)$. Thus, the isometry $gf$ falls in the previous case and our claim holds.
\item[$a,b\neq 0$ and $2| u$] By the same argument of the previous case we may suppose that $b=2c$.
In this case $a$ has to be odd. A primitive submultiple of $2u+a\zeta$ is of the form
$2u'+a'\zeta$ with $a'$ odd.  By Eichler's criterion (cf. Lem. \ref{lem:eichler}), there exists 
$g\in O^{+}(H^2(A,\ZZ)\oplus_{\perp} \ZZ\zeta)$ such that $g(2u'+a'\zeta)=2u''+a'\zeta$
with $u''$ not divisible by $2$. It follows that $gf$ falls in the previous case, hence our claim holds for $f$.
\end{itemize}

(2) This is a standard consequence of (1) and Markman's Hodge theoretic version of Verbitsky's Global Torelli.
If $X'$ and $X''$ are bimeromorphic there exists a Hodge isometry $H^{2}(X',\ZZ)\simeq H^{2}(X'',\ZZ)$ by \cite[Proposition 1.6.2]{OGW2}. Conversely, given a Hodge isometry $\phi:H^{2}(X',\ZZ)\rightarrow H^{2}(X'',\ZZ)$ and a parallel transport operator
$t:H^{2}(X'',\ZZ)\rightarrow H^{2}(X',\ZZ)$, since $Mon^{2}(X')=O^+(X')$ and $-1$ reverses the orientation of the positive cone of the lattice $H^{2}(X',\ZZ)$ (see Remark \ref{positive cone}),
either  $t\circ\phi\in Mon^{2}(X')$ or $-(t\circ\phi)\in Mon^{2}(X')$. Hence either $\phi$ or $-\phi$ is a Hodge isometry and a parallel transport operator: by Markman's Hodge theoretic Torelli theorem \cite[Theorem 1.3]{mark_tor} the \hk manifolds  $X'$ and $X''$ are birational. 
\end{proof}
\begin{oss} We have actually proven that  $Mon^{2}(\widetilde{K}_{(2,0,-2)}(A,H))$ consists of monodromy operators for families 
of projective manifolds. This statement holds for $\widetilde{K}_{v}(A,H)$ for every $v,A$ and a $v$-generic $H$ and follows
since  the proof of \cite[Theorem 1.7]{pr_crelle} shows that the \hk manifold   $\widetilde{K}_{v}(A,H)$  may be deformed to
$\widetilde{K}_{(2,0,-2)}(A,H)$ using only families of projective manifolds.
\end{oss}


\section{The K\"ahler cone and the birational K\"ahler cone}\label{sec:mov}
The aim of this section is to compute the \kahl and Birational \kahl cones for manifolds of $OG6$ type.   
In  general, for a \hk manifold $X$ these cones are subcones of the positive cone\footnote{ 
The positive cone of a \hk manifold $X$ is strictly contained in the positive cone $C(H^{2}(X,\ZZ))$
of the Beauville-Bogomolov lattice of $X$ as defined in in Subsection \ref{ssec:lattice} and Remark \ref{positive cone}.} $C(X)\subset H^{1,1}(X,\mathbb{R})$ that is defined as the connected component of the set of classes with positive Beauville-Bogomolov
square containing a \kahl class. The Birational \kahl cone is the union $\bigcup_f f^*(K(X'))$ of \kahl cones of $X'$, where $f\,:\,X\dashrightarrow X'$ runs on all birational maps between $X$ and other \hk manifolds $X'$.
The \kahl and the birational \kahl cone  are dual in $C(X)$ to wall divisors and stably prime exceptional divisors respectively (\cite[Section 6]{mark_tor} and \cite[Proposition 1.5]{mon_kahl}). The main result of this section says  that the \kahl and birational \kahl cones of a manifold of $OG6$ type are completely determined by the lattice structure on the Picard lattice.

Before discussing \hk manifolds of $OG6$ type, we recall  definitions and basic properties of stably  prime exceptional divisors and wall divisors.
\begin{defn} A prime exceptional line bundle on a \hk manifold $X$ is a line bundle associated with a  
prime exceptional divisor on $X$.   A stably prime exceptional divisor $D\in Div(X)$ is an effective  divisor whose associated line bundle  is prime exceptional in a general deformation of the pair $(X,\mathcal{O}(D))$.
\end{defn} 

Prime exceptional divisors are stably prime exceptional divisors (see \cite[Proposition 6.6(1)]{mark_tor}), but the converse does not hold in general. The easiest  example of a stably prime exceptional divisor which is not prime exceptional is given by a reducible $-2$ curve on a K3 surface.
 
By looking at the configuration of orthogonal hyperplanes to stably prime exceptional divisors, the positive cone 
 $C(X)$ is cut in a wall and chamber decomposition. One such chamber is the convex hull of the Birational \kahl cone (see \cite[Section 5.2]{mark_tor}) and its algebraic part is the interior of the movable cone, i.e. the cone generated (over $\mathbb{R}$) by all divisors which do not have a divisorial base locus (see \cite{bou} and \cite[Theorem 5.8]{mark_tor}). We remind the reader that the Birational \kahl cone is not connected in general.

\begin{defn} A wall divisor on a \hk manifold $X$ is a  divisor $D$ of negative Beauville-Bogomolov square and such that, for every 
monodromy operator $g:H^{2}(X,\ZZ)\rightarrow H^{2}(X,\ZZ)$ that is a Hodge isometry, the perpendicular to
the class $g([D])$ does not intersect the Birational \kahl cone.   
\end{defn}
By using the natural lattice embedding $H_2(X,\mathbb{Z})\hookrightarrow H^2(X,\mathbb{Q})$, a class of a wall divisor is precisely a multiple of a class of an extremal rational curve, up to the action of monodromy Hodge isometries, see \cite[Proposition 2.3]{KLM2}. Orthogonal hyperplanes to wall divisors also give a wall and chamber decomposition of the positive cone (whence their name), and one of the open chambers is the \kahl cone. In particular, if we restrict this wall and chamber decomposition to the birational \kahl cone, we obtain the \kahl cones of all \hk birational models of $X$.\\
Every stably prime exceptional divisor is a wall divisor 
but the converse does not hold and wall divisor which are not multiples of stably prime exceptional divisors are those responsible for the non connectedness of the birational \kahl cone.

\begin{oss}\label{moninv} 
A very important  property of these classes of divisors is  their invariance under parallel transport:  if $X,Y$ are \hk and $D\in Div(X)$ is a wall divisor (resp. stably prime exceptional), and $\varphi$ is a parallel transport operator between $X$ and $Y$ such that $\varphi([D])\in Pic(Y)$, then $\varphi([D])$ is the class of a wall divisor (resp. up to sign the class of a stably prime exceptional), see \cite[Theorem 1.3]{mon_kahl} (resp. \cite[Proposition 5.14]{mar_prime}).
 Therefore, in order to describe the walls, it suffices to determine the classes of stably prime exceptional and wall divisors up to parallel transport. 
\end{oss}

To determine classes 
of stably prime exceptional divisors and wall divisors  in the case of \hk manifolds of 
$OG6$ type, we will use two tools. The first is the  birational geometry  
of O'Grady six dimensional manifolds, to prove that some divisors are either stably prime exceptional or wall divisors. 
The second is  
the construction  of ample divisors on Albanese fibres of moduli space of sheaves to prove that some divisors are not wall divisors.
Let us start with the first approach:
\begin{lem}\label{lem:pex}
Let $X$ be a manifold of $OG6$ type. Let $D\in Div(X)$ , let $[D]\in H^{2}(X,\ZZ)$ be its class and let $div(D)$ be the divisibility of $[D]$ in $H^{2}(X,\ZZ)$. If one of the following holds:
\begin{itemize}
\item $[D]^2=-4$ and $div(D)=2$,
\item $[D]^2=-2$ and $div(D)=2$,
\end{itemize}
then $[D]$ 
is a primitive class proportional to the class of a stably prime exceptional divisor. 
\end{lem}
\begin{proof}
Let $A$, $H$ and $v$ be as in Setting \ref{set}, assume that  $v:=(2,0,-2)$ and set $X:=\widetilde{K}_v(A,H)$.
We have two prime effective divisors $\widetilde{\Sigma}$ and $B$
on $X$, which are respectively the exceptional divisor of the resolution $\pi\,:\,X:=\widetilde{K}_v(A,H)\rightarrow K_v(A,H)$ and the strict transform of the locus parametrizing non locally free sheaves. By \cite[Prop. 3.3.2]{rap_phdz}, there exists a divisor $E$ such that $2[E]=[\widetilde{\Sigma}]$. By \cite{per} and \cite[Theorem 3.5.1]{rap_phdz} we have  $div([E])=div([B])=2$,
$[E]^2=-2$ and $[B]^2=-4$:
 in particular 
$[\widetilde{\Sigma}]$ and  
$[B]$ are classes of prime exceptional divisors, hence also stably prime exceptional, and by Remark \ref{moninv} every element in the parallel transport orbits  of $[E]$ and $[B]$
is proportional to the class of 
a stably prime exceptional divisor. By Lemma \ref{lem:eichler} and Theorem \ref{thm:monodromy}, the square and divisibility determine the orbits of these classes. 
\end{proof}
To construct wall divisors whose classes  are not multiples of classes of stably prime exceptional divisors we use a small contraction  of 
a specific \hk manifold of $OG6$ type whose singular model has already been investigated in Corollary \ref{cor:cham}. We discuss the needed birational geometry in the following example.

\begin{ex}\label{esempio}
Let $A=E_1\times E_2$ be the product of two  elliptic  curves $E_1,E_2$. Let their classes $e,f$ generate  $NS(A)\cong U$ and let $H,H_0$ be two ample line bundles whose image in $NS(A)$ are respectively $de+f,e+f$ with $d>5$.   
By  \cite[Definition 2.1]{PerRap} the line bundle  $H$ is a $(2,0,-2)$-generic  polarization  while $H_0$ is not a $(2,0,-2)$-generic polarization.
As in the $v$-generic case, we denote by $M_{(2,0,-2)}(A,H_0)$ the moduli space of $H_0$-semistable sheaves on $A$ with Mukai vector    
$(2,0,-2)$.

By Proposition \ref{prop:cham} every $H$-semistable sheaf with Mukai vector $(2,0,-2)$ is also $H_0$-semistable, hence there is a natural regular morphism $$\overline{c}:  M_{(2,0,-2)}(A,H)\rightarrow M_{(2,0,-2)}(A,H_0)$$ and we let   
$K_{(2,0,-2)}(A,H_0)$ be the image under $f$ of $K_{(2,0,-2)}(A,H)$ or, equivalently, the Albanese fiber of $M_{(2,0,-2)}(A,H_0)$.
The restriction of $\overline{c}$ gives 
a small contraction $$c:K_{(2,0,-2)}(A,H)\rightarrow K_{(2,0,-2)}(A,H_0).$$ By Proposition \ref{prop:cham} the contracted locus of $c$ represents sheaves that are non $H'$-semistable for an ample line bundle $H'$ whose class is $e+df$ and, by Corollary \ref{cor:cham}, it consists of a finite set of copies of $\mathbb{P}^3$ contained in the smooth locus of $K_{(2,0,-2)}(A,H)$.


Set $X:=\widetilde{K}_{(2,0,-2)}(A,H)$  and let $X_0$
be the blow up of $K_{(2,0,-2)}(A,H_0)$ along the image under $c$ of the singular locus of $K_{(2,0,-2)}(A,H)$. 
We have a commutative diagram 
\begin{equation}
\xymatrix{
X \ar[d] \ar[r]^{\widetilde{c}} & X_0 \ar[d]\\
K_{(2,0,-2)}(A,H) \ar[r]^{c} 
& K_{(2,0,-2)}(A,H_0) & 
}
\end{equation}
and $\widetilde{c}$ is the small contraction we are interested in. We will use  that 
$\widetilde{c}$ is a relative Picard rank one
contraction and the class  $[D]\in H^{2}(K_{(2,0,-2)}(A,H),\ZZ)$ of the contracted extremal curve is  proportional to $ (0,e-f,0)$.

This can be proved as follows.
By the Hodge isometry $H^{2}(X,\ZZ)\simeq (2,0,-2)^{\perp}\oplus_{\perp} \ZZ\epsilon $
(see equation (\ref{Hiso})), the  rank of the Picard group of $X$ is four. Since, by construction, the exceptional divisor $\widetilde{\Sigma}$ of the blow up $X\rightarrow K_{(2,0,-2)}(A,H)$ does not intersect the contracted locus of $\widetilde{c}$, it descends to a Cartier divisor on $X_0$.
Moreover (see the proof of \cite[Proposition 3.9]{pr_crelle}),  the restriction of the Mukai-Donaldson-Le Potier morphism to the algebraic part of $(1,0,-1)^{\perp}$ is defined as in  \cite[Theorem 8.1.5]{HL} and, by this theorem, every  class belonging to the lattice  $(1,0,-1)^{\perp}= H^{2}(K_{(2,0,-2)}(A,H),\ZZ)$, whose component in $H^{2}(A,\ZZ)$ is proportional to $e+f$, descends to a class of a line bundle in $H^{2}(K_{(2,0,-2)}(A,H_0),\ZZ)$.
It follows that $\widetilde{c}$ is a relative Picard rank one
contraction and the saturation of $Pic(X_0)$ inside $Pic(X)$ is generated by $\epsilon$, $e+f$ and $(1,0,1)$.
Finally   $[D]$ is perpendicular to  $Pic(X_0)$, hence proportional to $ (0,e-f,0)$.

\end{ex}

\begin{lem}\label{lem:wall}
Let $X$ be a manifold of $OG6$ type. 
Let $D\in Div(X)$ , let $[D]\in H^{2}(X,\ZZ)$ be its class and let $div(D)$ be the divisibility of $[D]$ in $H^{2}(X,\ZZ)$.  
If $[D]^2=-2$ and $div(D)=1$, then $[D]$ is the class of a wall divisor but not the class of a multiple of a stably prime exceptional divisor.
\end{lem}
\begin{proof} Since $Mon^{2}(X)\simeq O^+(U^3\oplus_{\perp}(-2)^2)$ by Theorem \ref{thm:monodromy}, by Remark \ref{moninv} and Lemma \ref{lem:eichler} it suffices to show the existence of a specific \hk manifold $X$
of $OG6$ type and a wall divisor $D\in Div(X)$ such that $D$ is not a stably prime exceptional divisor and
$[D]^2=-2$ and $div([D])=1$.
We consider the case where $X$
is as in Example \ref{esempio} and $[D]=(0,e-f,0)\in H^{2}(X,\ZZ)$: clearly $[D]^2=-2$ and $div(D)=1$. 
As shown in Example \ref{esempio}
the class  $[D]$ is, up to scalars, the image  in $H^{2}(X,\mathbb{Q})$ of the class of an extremal  curve giving a small contraction: by \cite[Proposition 2.3]{KLM2} this implies that $D$ is a wall divisor and no multiple of $[D]$ is the class of a stably prime exceptional divisor. 
\end{proof}

Let us move to the second part of the proof. We exhibit ample line bundles on \hk manifolds of $OG6$ type
by using a classical construction that produces ample line bundles on the the loci of moduli spaces of sheaves 
where the determinant is fixed. 

To produce the ample line bundles we consider again the case of Example \ref{esempio}. 
Before stating the result we recall that, by Corollary \ref{cor:cham}, if the class of the $(2,0,-2)$-generic polarization $H$ is 
$h=a_he+b_hf$, the Albanese fiber $K_{(2,0,-2)}(A,H)$ of the moduli space of $H$-semistable sheaves with Mukai vector 
$(2,0,-2)$ only depends on the sign of $a_h-b_h$.
\begin{lem}\label{lem:ample}
Let $A=E_1\times E_2$ be the product of two elliptic curves such that  $NS(A)$ is generated by the classes $e$ and $f$ of $E_1$ and $E_2$.
Let $H$ be a $(2,0,-2)$-generic polarization,  let $h:=a_he+b_hf\in NS(A)$ be its class and assume that $a_h>b_h$.  
For $a>b>0$ set $\gamma:=ae+bf\in NS(A)$.   
\begin{enumerate}
\item{Using the identification provided by the Mukai-Donaldson-Le Potier morphism (\ref{mdl*}), for $k>>0$, the class
$$(-1,k\gamma ,-1)\in (1,0,-1)^{\perp}=H^{2}(K_{(2,0,-2)}(A,H),\ZZ)$$ is the class of an ample line bundle on $K_{(2,0,-2)}(A,H)$.}
\item{Using the identification (\ref{Hiso}), for $k>>c>>d>0$, the class $$(-c,k\gamma,-c)+ (-d)\epsilon \in (1,0,-1)^{\perp}\oplus_{\perp} \ZZ \epsilon =H^{2}(\widetilde{K}_{(2,0,-2)}(A,H),\ZZ)$$ is the class of an ample line bundle on $\widetilde{K}_{(2,0,-2)}(A,H)$.}
\item{Using the identification (\ref{Hiso}), there exist integral numbers $c,m,d>0$ such that 
the class $$(-c,(m+1)e+mf,-c)+ (-d)\epsilon \in (1,0,-1)^{\perp}\oplus_{\perp} \ZZ \epsilon =
H^{2}(\widetilde{K}_{(2,0,-2)}(A,H),\ZZ) $$ is the class of an ample line bundle on $\widetilde{K}_{(2,0,-2)}(A,H)$.}
\end{enumerate}
  \end{lem}   
\begin{proof}
(1) By convexity of the ample cone of $K_{(2,0,-2)}(A,H)$, it suffices to prove the result in the case where 
$\gamma$ is the class of a $(2,0,-2)$-generic polarization. In this case, by Corollary \ref{cor:cham}, we may assume that 
$\gamma =h$.  
By \cite[Theorem 8.1.11]{HL} the class $(-1,0,-1+k^2h^2/2)$ is the class of an ample divisor
on the Albanese fibre $K_{(2,2kh, -2+ k^2h^2)}(A,H)$ of the moduli space of $H$-semistable sheaves on $A$
with Mukai vector $(2,2kh, -2+ k^2h^2)$. Since tensoring by a multiple of the polarization preserves  stability
and semistability and is compatible with the Mukai-Donaldson-Le Potier morphism, tensoring back 
by $-kH$ we get that $(-1,kh,-1)\in (1,0,-1)^{\perp}$ is the class of an ample line bundle on   
$K_{(2,0,-2)}(A,H)$.

(2) Since $\pi_{(2,0,-2)}:\widetilde{K}_{(2,0,-2)}(A,H)\rightarrow K_{(2,0,-2)}(A,H)$ is a divisorial contraction of an extremal curve such that the class of the contracted divisor  is a positive multiple of $\epsilon$, item (1) implies (2).

(3) By (2), if $a\ge b\ge 0$, the class $(0,ae+bf,0)\in  H^{2}(\widetilde{K}_{(2,0,-2)}(A,H),\ZZ)$ is a limit of classes of ample line bundles: hence it is the class of a nef line bundle and the same holds for the classes  $(0,e+f,0)$ and $(0,e,0)$. By Example \ref{esempio}, the class of a line bundle on  $\widetilde{K}_{(2,0,-2)}(A,H)$ descending to an ample line bundle on $X_0$
is of the form $(-c,m(e+f),-c) +(-d)\epsilon$: moreover, since $(0,e+f,0)$ is nef and positive multiples of $(1,0,1)$ and
$\epsilon$ are effective with negative Beauville-Bogomolov square (see \cite[Theorem 9.1]{per}), we get $m,c,d>0$. Since $X_0$ is the contraction of an extremal curve of class proportional to $(0,-e+f,0)$ and $(0,e,0)$ is nef with degree non zero on that curve (see Example \ref{esempio}), we obtain that
$(-c,m(e+f)+e,-c) +(-d)\epsilon$
is the class of an ample divisor on $\widetilde{K}_{(2.0.-2)}(A,H)$.
\end{proof}

With the above, we are now ready to determine wall divisors and stably prime exceptional divisors on \hk manifolds of $OG6$ type:
\begin{prop}\label{prop:wall}
Let $X$ be a manifold of $OG6$ type. 
Let $D\in Div(X)$ , assume that its class  $[D]\in H^{2}(X,\ZZ)$ is primitive and let $div(D)$ be the divisibility of $[D]$ in $H^{2}(X,\ZZ)$.
 Then $D$ is a wall divisor if and only if one of the following holds:
\begin{enumerate}[i)]
\item{ $[D]^2=-4$ and $div([D])=2$,} 
\item{ $[D]^2=-2$ and $div([D])=2$,} 
\item{ $[D]^2=-2$ and $div([D])=1$.} 
\end{enumerate}
In cases $i)$ and $ii)$ a multiple of $[D]$ is the class of a stably prime exceptional and in case $iii)$ no multiple of $[D]$ is represented by a non zero effective divisor.
\end{prop}
\begin{proof}
As wall divisors are invariant under parallel transport which preserves their Hodge type by \cite[Theorem 1.3]{mon_kahl}, we can prove the statement on a specific \hk manifold of $OG6$ type: we consider the case  $X:=\widetilde{K}_{(2,0,-2)}(A,H)$ as in Lemma \ref{lem:ample}. 
Then, $Pic(X)\cong U\oplus_{\perp} (-2)^2$ and by Lemma \ref{lem:eichler} and Theorem \ref{thm:monodromy}, any class in $H^2(X,\mathbb{Z})$ can be moved with the monodromy group inside $Pic(X)$, therefore every wall divisor shows up in this case.

By Lemmas \ref{lem:pex} and \ref{lem:wall}, elements of square $-2$ or $-4$ and divisibility $2$ are indeed   classes of wall divisors with an effective multiple (that is, stably prime exceptional divisors) and elements of square $-2$ and divisibility $1$ are classes of wall divisors with no effective multiples (that is, wall divisors which are not
stably prime exceptional).
This proves the 'if' part of the statement.

By Lemma \ref{lem:eichler}, since the order of the discrimnant group of the Beauville Bogomolov lattice is $4$, beyond the cases listed in $i)$, $ii)$ and $iii)$ we have four standard forms 
for classes of primitive divisors $[D]\in H^{2}(X,\mathbb{Z})=(1,0,-1)^{\perp}\oplus \ZZ\epsilon$ of strictly negative squares: 
\begin{enumerate}[A)]
\item{$[D]=(0,ae-f,0)$ with $a>1$,}
\item{$[D]=(-1,2(-ae+f),-1)-\epsilon$ with $a\ge 1$,}
\item{$[D]=(-1,2(-ae+f),-1)$ with $a\ge 1$,}
\item{$[D]=(0,2(-ae+f),0)-\epsilon$ with $a\ge 1$.}
\end{enumerate} 
By Theorem \ref{thm:monodromy}(1),  every primitive class in the Picard group of  a \hk manifold of $OG6$ type  of negative square  and not satisfying i),ii) or iii), 
can be moved by a parallel transport operator to a class $[D]\in Pic(X)$ as in A), B), C) or  D). 
Since the image under a parallel transport operator of the class of a wall divisor is a  class of a wall divisors if it is of Hodge type (see Remark \ref{moninv}), 
we have to prove that every class 
as in A), B), C) or D) is not the class of a wall divisor.
Since cases C) and D) are monodromy equivalent, it suffices to show that every class as in   A), B), or C) is not the class of a wall divisor.  

In case A), let $\gamma=ae+f\in NS(A)$ be a primitive ample class. Let $ae-f$ be a generator of $\gamma^\perp$ in $NS(A)$.   
By  Lemma \ref{lem:ample}(2), $(-c,k\gamma,-c)+(-d)\epsilon$ is ample on $\widetilde{K}_v(A)$ for $k>>c>>d>0$.
This ample divisor is orthogonal to $(0,ae-f,0)$: therefore $(0,ae-f,0)$ is not the class of a wall divisor.

To deal with case B), we use again Lemma \ref{lem:ample}.
By Lemma \ref{lem:ample}(3) there exists an ample divisor on $X$ whose class is 
$(-c, (m+1)e+mf, -c)+(-d) \epsilon$ for strictly positive $m,c,d\in \ZZ$.
By Lemma \ref{lem:ample}(1), the class $(0,e,0)$ is limit of classes of ample divisors, hence it is the class of a nef line bundle. In particular, since  $ma-m+c+d-1\ge0$, the class 
$$\Gamma:=(-c, (m+1)e+mf, -c)+(-d) \epsilon + (0,(ma-m+c+d-1)e,0)= $$$$=(-c,(ma+c+d)e+mf, -c) +(-d) \epsilon=
(c,m(ae+f)+(c+d)e, -c)+(-d) \epsilon.$$
is the sum of classes of an ample and a nef divisor, hence it is the class of an ample divisor on $X$.
Since $(-1,2(-ae+f),-1)-\epsilon$ is perpendicular to $\Gamma$ we conclude that 
it is not the class of a wall divisor.

In case C) one can argue as in case B) and show that
the class $$\Gamma:= (-c,(m+1)e+mf, -c)+(-d) \epsilon + (0,(ma-m+c-1)e,0)= $$$$=(-c,(ma+c)e+mf, -c)+(-d) \epsilon=
(-c,m(ae+f)+ce, -c)+(-d) \epsilon$$
is the class of an ample divisor on $X$ perpendicular to $(-1,2(-ae+f),-1)$. 
\end{proof}
As a consequence of Proposition \ref{prop:wall}, using \cite[Section 6]{mark_tor} and \cite{mon_kahl}, we get the main result of this section.
In the statement, following the standard notation, for every \hk manifold $X$ and for every $\alpha\in H^2(X,\ZZ)$, we denote by 
$\alpha^{\perp_{B_{X}}}$ the perpendicular to $\alpha\in H^2(X,\mathbb{R})$ with respect to the real extension of the Beauville-Bogomolov form and we denote by $div(\alpha)$
the divisibility of $\alpha$ in the lattice  $H^2(X,\mathbb{Z})$ (see Definition \ref{def1}).

\begin{thm}\label{thm:amp}
Let $X$ be a \hk manifold of $OG6$ type and 
let $C(X)$ be the positive cone of $X$. Then
\begin{enumerate}
\item{The closure in $H^{1,1}(X,\mathbb{R})$ of the birational K\"ahler cone $\overline{BK}(X)$ of $X$ is  the closure of the connected component of $$C(X)\setminus \bigcup_{\substack{\alpha\in H^{1,1}(X,\ZZ),\\
B_{X}(\alpha,\alpha)=-2\; or \; -4, \\ div(\alpha)=2.}} \alpha^{\perp_{B_{X}}}$$ containing a K\"ahler class.}
\item{The K\"ahler cone $K(X)$ is the connected component of 
$$C(X)\setminus \bigcup_{\substack{\alpha\in H^{1,1}(X,\ZZ),\\
B_{X}(\alpha,\alpha)=-2\; or  \\ 
B_{X}(\alpha,\alpha)=-4\; and\;  div(\alpha)=2.}} \alpha^{\perp_{B_{X}}}$$
containing a K\"ahler class.}
\end{enumerate}

\end{thm}
\begin{proof}
(1) Let $S$ be the set of stably prime exceptional divisors on $X$. By \cite[Proposition 6.10]{mark_tor}, the closure of the birational K\"ahler cone is the closure of the component of $C(X)\setminus\bigcup_{\alpha\in S}\alpha^{\perp_{B(X)}}$ containing a \kahl class. By Proposition \ref{prop:wall}, stably prime exceptional divisors are (up to multiples) those of divisibility $2$ and squares $-2$ or $-4$, so the claim follows. \\
(2) Let $W$ be the set of wall divisors on $X$. Analogously, by \cite[Proposition 1.5]{mon_kahl}, the \kahl cone is the connected component of $C(X)\setminus\bigcup_{\alpha\in W}\alpha^{\perp_{B(X)}}$ containing a \kahl class. By Proposition \ref{prop:wall}, wall divisors are those of square $-4$ and divisibility $2$ or of square $-2$ and any divisibility, therefore the claim follows.
\end{proof}
\begin{oss}
For every \hk manifold $X$ the ample cone can be obtained by intersecting the \kahl cone with $H^{1,1}(X,\mathbb{Q})\otimes \mathbb{R}$ and
similarly the movable cone coincides with the intersection of the closure of the birational K\"ahler cone with
$H^{1,1}(X,\mathbb{Q})\otimes \mathbb{R}$, as long as nef divisors are also movable. The latter is a consequence of Theorem \ref{prop:lagr}. Therefore, Theorem \ref{thm:amp}
also determines the ample and the movable cones of every \hk manifold of
$OG6$ type.
\end{oss}
\section{Lagrangian fibrations and applications}\label{sec:lagr}
The aim of this section is to prove that, whenever a manifold of $OG6$ type has a square zero divisor, it has a rational lagrangian fibration. First, we establish the number of monodromy orbits of a square zero divisor:
\begin{lem}\label{lem:lagrorbit}
Let $l\in L:=U^3\oplus (-2)^2$ be a primitive  element of square zero. Then $div(l)=1$ and there is a single orbit for the action of $O^+(L)$ on the set of  primitive  isotropic elements of $L$.
\end{lem}
\begin{proof}
As the discriminant group of $L$ is of two torsion, the divisibility can be either one or two. 
Any primitive element of divisibility $2$ can be written as $2w+at+bs$ for some $w\in U^3$ and $t,s$ such that $\langle t,s\rangle=(-2)^2$ and $ \langle t,s\rangle^\perp=U^3$, moreover if  $a$ is odd  $b$ is  even. This means that, modulo $8$, the square of such an element is congruent to either $-2$ or $-4$, which is not the case for an element of square zero. Therefore $l$ has divisibility one and, by lemma \ref{lem:eichler}, the action of $O^+(L)$ has a single orbit.
\end{proof}
Recall that, for any \hk manifold $X$, if $p\,:\, X\to \mathbb{P}^n$ is a lagrangian fibration, the divisor $p^*(\mathcal{O}(1))$ is primitive, nef and isotropic. In particular, if a divisor is induced by a lagrangian fibration on a different birational model of $X$, it will be isotropic and in the boundary of the Birational \kahl cone. The following is a converse for manifolds of $OG6$ type:

\begin{thm}\label{prop:lagr}
Let $X$ be a \hk manifold of $OG6$ type and let $\mathcal{O}(D)\in Pic(X)$ be a non-trivial line bundle whose Beauville-Bogomolov  square is $0$.
Assume that the class $[D]$ of $\mathcal{O}(D)$ belongs to the boundary of the birational K\"ahler cone of $X$. 
Then, there exists  a smooth \hk manifold $Y$ and a bimeromorphic map 
$\psi\,:Y\to X$ such that $\mathcal{O}(D)$ induces  a lagrangian fibration $p:Y\rightarrow \mathbb{P}^3$. 
Moreover, smooth fibres of $p$ are $(1,2,2)$-polarized abelian threefolds.
\end{thm}
\begin{proof}
By the work of Matsushita \cite[Theorem 1.2]{matsu_lagr} the locus $V_{bir}(X,\mathcal{O}(D))$, inside the base $Def(X,\mathcal{O}(D))$ of the universal deformations  of the pair $(X,\mathcal{O}(D))$, where the parallel transport of $[D]$ defines a birational lagrangian fibration is either the locus $V_{mov}(X,\mathcal{O}(D))$ where the parallel transport of $[D]$ belongs to the boundary of the birational k\"ahler cone or empty. Moreover if the locus $V_{bir}(X,\mathcal{O}(D))$ is non empty, it is dense  in $Def(X,\mathcal{O}(D))$.
In particular, if we have a family of pairs $(\mathcal{X},\mathcal{D})\rightarrow B$, the locus 
$$ \{t\in B\,\text{ such that } V_{bir}(X_t,\mathcal{O}(D_t)) \text{ is dense in } Def(X_t,\mathcal{O}(D_t))\}$$
is open and closed. Therefore, the statement holds for $(X,\mathcal{O}(D))$ if and only if it holds for a deformation $(X',\mathcal{O}(D'))$ of it 
such that 
$[D']$ belongs to the closure of the birational \kahl cone.

It remains to show that the space parametrizing pairs of the form $(X',\mathcal{O}(D'))$, such that  $[D']^2=0$ and $[D']$ belongs to the boundary of the birational K\"ahler cone of $X'$,  is connected  and that, for one such pair,  $\mathcal{O}(D')$ defines a birational lagrangian fibration.

By Lemma \ref{lem:lagrorbit}, for every $(X',\mathcal{O}(D'))$,  there exists a parallel transport operator sending $[D]\in H^2(X,\mathbb{Z})$ to $[D']\in H^2(X',\mathbb{Z})$. Moreover  since $[D]$ and $[D']$ are in the boundaries of the birational K\"ahler cones, 
there are  K\"ahler classes $k$ and $k'$, on $X$ and $X'$ respectively, that are positive on $[D]$ and $[D']$
with respect to the Beauville Bogomolov pairings. By \cite[Lemma 5.17(ii)]{mar_prime} we obtain  that $(X',\mathcal{O}(D'))$ and  $(X,\mathcal{O}(D))$ are deformation equivalent.

Finally one example where a lagrangian fibration exists is well known. It suffices to consider $\widetilde{K}_{(0,2h,0)}(A,H)$ for a very general principally polarized Abelian surface $(A,h)$: the Lagrangian fibration associated to it is the one induced by the fitting morphism, sending a sheaf into its support. These fibres have étale double covers which are Jacobians of genus three curves (see \cite{rap_phd} and \cite[Remark 5.1]{MRS}), and have the required polarization of type $(1,2,2)$ by \cite[Corollary 12.1.5]{Birkenhake-Lange}. Moreover, by \cite[Theorem 1.1]{wie}, the polarization type of a general fibre of a birational lagrangian fibration induced by an isotropic divisor $D$ on a \hk $X$ only depends on the deformation class of the pair $(X,\mathcal{O}(D))$.
\end{proof}
As a consequence we settle a conjecture due to Hassett-Tschinkel, Huybrechts and Sawon  for this deformation equivalence class: 
\begin{cor}\label{cor:lagr}
Let $X$ be a manifold of $OG6$ type with a square zero non zero divisor. Then $X$ has a bimeromorphic   model which has a dominant map to $\mathbb{P}^3$ whose general fiber is a $(1,2,2)$-polarized abelian threefold.
\end{cor}
\begin{proof} By   Theorem \ref{prop:lagr} it suffices to show that there exists another isotropic divisor on $X$ which is in the boundary of the birational \kahl cone. This follows from \cite[Section 6]{mark_tor}, where Markman proves that $Mon^2(X)\cap Hdg(H^2(X))$ (the group of Hodge isometries which are monodromy operators) acts transitively on the set of exceptional chambers of the positive cone, one of which contains the Birational \kahl cone and has its same closure (moreover, every element of the closure of the positive cone is in the closure of one such exceptional chamber). Thus, either $[D]$ or $-[D]$ is in the closure of one such exceptional chamber, and a monodromy Hodge isometry can move it to the boundary of the birational \kahl cone.
\end{proof}

As an  immediate consequence of the above Theorem, by using \cite[Theorem 4.2]{rie}, we obtain that the Weak Splitting property conjectured by Beauville \cite{beau07} holds when the manifold has a non zero square zero divisor.
\begin{cor}\label{ultimo}
Let $X$ be a projective  \hk  manifold of $OG6$ type and let $D\neq 0$ be a square zero divisor on it. Let $DCH(X)\subset CH_{\mathbb{Q}}(X)$ be the subalgebra generated by divisor classes. Then the restriction of the cycle class map $cl_{|DCH(X)}\,:\,DCH(X)\to H^*(X,\mathbb{Q})$ is injective .
\end{cor}
\begin{proof}
Theorem 4.2 of \cite{rie} proves that the Weak Splitting property holds for all manifolds $X$ such that one of their birational model has a lagrangian fibration
and Corollary \ref{cor:lagr} prove that this lagrangian fibration exists for any manifold of $OG6$ type with a non zero isotropic divisor.   
\end{proof}

\end{document}